\documentclass[abstract=on,headings=small,11pt]{scrartcl}

\usepackage[utf8]{inputenc}
\usepackage[T1]{fontenc}
\usepackage{graphicx}
\usepackage[final]{microtype}
\usepackage{longtable}
\usepackage{wrapfig}

\usepackage{mathtools}
\newcommand{\N}{\ensuremath{\mathbf{N}}}
\newcommand{\Z}{\ensuremath{\mathbf{Z}}}
\newcommand{\R}{\ensuremath{\mathbf{R}}}

\DeclareMathOperator{\id}{id}

\usepackage[paperwidth=165mm, paperheight=242mm, left=22.5mm, right=22.5mm, top=20mm, bottom=20mm]{geometry}

\usepackage{textcomp}
\usepackage{amssymb}
\usepackage{lmodern}
\usepackage{amsthm}
\usepackage{enumitem}
\PassOptionsToPackage{hyperref}{hidelinks}
\usepackage{csquotes}
\usepackage[british]{babel}
\usepackage[backend=biber,doi=false,isbn=false,url=false,maxbibnames=9,maxcitenames=2]{biblatex}
\bibliography{main}

\usepackage{pgfplots}
\usepackage{tikz}
\usetikzlibrary{patterns,tikzmark,calc}
\pgfplotsset{%
  ,compat=1.12
  ,every axis x label/.style={at={(current axis.right of origin)},anchor=north west}
  ,every axis y label/.style={at={(current axis.above origin)},anchor=north east}
}
\usepackage{tikz-cd}

\newtheorem{theorem}{Theorem}
\newtheorem{definition}{Definition}
\newtheorem{corollary}{Corollary}
\newtheorem{lemma}{Lemma}

\MakeAutoQuote{«}{»} 
\MakeAutoQuote*{‹}{›} 

\DeclareMathOperator{\Prob}{\mathbb{P}}
\DeclareMathOperator{\Expect}{\mathbb{E}}
\DeclareMathOperator{\PBin}{PBin}
\DeclareMathOperator{\Bin}{Bin}
\DeclareMathOperator{\Mult}{Mult}

\title{Information in additional observations of a non-parametric experiment
  that is not estimable}
\author{Tilo Wiklund}
\date{\vspace{-1em}}

\begin{document}

\maketitle{}

\begin{abstract}
  Given \(n\) independent and identically distributed observations and measuring
  the value of obtaining an additional observation in terms of Le~Cam's notion
  of deficiency between experiments, we show for certain types of non-parametric
  experiments that the value of an additional observation decreases at a rate of
  \(1/\sqrt{n}\). This is distinct from the known typical decrease at a rate of
  \(1/n\) for parametric experiments and the non-decreasing value in the case of
  very large experiments. In particular, the rate of \(1/\sqrt{n}\) holds for the
  experiment given by observing samples from a density about which we know only
  that it is bounded from below by some fixed constant. Thus there exists an
  experiment where the value of additional observations tends to zero but for
  which no estimator that is consistent (in total variation distance) exists.
\end{abstract}

\section{Introduction}

Much of traditional statistics is concerned with studying the performance of
statistical procedures or the difficulty of statistical decision problems as a
function of the number of observed independent repetitions of some experiment.
Considering only the repetition of the experiment, without reference to a
particular decision problem, leads one to study the deficiency of
Le~Cam\cite{LeCam1986} between experiments differing only in the number of
independent repetitions. Roughly speaking, this corresponds to measuring the
difference in risks with respect to the decision problem that makes this
difference the greatest. This is contrary to the more common issue of determining how difficult a particular decision problem is, for example by establishing minimax bounds on risk, or determining the performance of a particular decision procedure or sequence of decision procedures. Rather, we are interested in the «totality» of information in the experiment. The rate at which this information changes tells us something of the inherent complexity of the experiment as well as indicating at what point further repetitions of an experiment yield diminishing returns.

Fairly little appears to be known concerning the possible rates at which the value, in the above sense, of having one more observation changes. Helgeland\cite{Helgeland1982} and Mammen\cite{Mammen1986} have
shown that for experiments that are, in an appropriate sense, finite-dimensional,
the value decreases at a rate of \(1/n\), where \(n\) is the number of
observations. The result of Mammen covers, in particular, finite-dimensional
exponential families. It is also known that when the parameter set is finite, the value of additional observations decreases at an exponential rate\cite{Torgersen1981}.

On the other hand, it is not too difficult to construct experiments where this
difference does not decrease at all. This means that for every \(n\) there is
some question that was impossible to solve based on \(n\) observations that
becomes trivial based on \(n+1\) observations. Examples of such experiments
include the experiment given by observing a completely unknown
density on the unit interval or the experiment given by observing a uniformly
chosen element from an unknown finite subset of a fixed infinite set.

The strategy used by both Helgeland and Mammen assumes the existence of a
sufficiently good estimator of the underlying, unknown, probability measure from
which samples are drawn. An upper bound on the deficiency is constructed from
the fact that it is then difficult to tell the difference between receiving a
genuine sequence of \(n+1\) independent observations as opposed to receiving
\(n+1\) observations out of which \(n\) are genuine independent and one, in a
random position, is a synthetic value sampled according to this \emph{estimate}
of the underlying distribution. If the estimator is sufficiently good, the bound
decays on the order \(1/n\). This bound was recently used by the author to study
the information content that is lost when re-sampling\cite{Wiklund2018}.

We will here consider what happens if this strategy is applied using the trivial
estimator that always estimates the underlying probability measure to be some
single, fixed, measure. Under some circumstances, this turns out to be enough to get
an upper bound of order \(1/\sqrt{n}\). Interestingly, this is precisely the
factor by which the technique of Helgeland and Mammen improves the bound
compared to a previous result of Le~Cam\cite{LeCam1974}.

We prove also that the rate \(1/\sqrt{n}\) appears as a lower bound on the
deficiency for experiments defined by observing samples according to an unknown
measure from a class of potential densities that is very large in a sense
similar to the \emph{rich} classes of
Devroye~\&~Lugosi\cite[Section~15.3]{Devroye2001}. In Devroye~\&~Lugosi these
classes appear as examples for which no consistent estimator, in total variation
distance, can exist. There is thus a certain parallel to the technique used to
establish the upper bound, where the technique requires the existence of a
non-trivial estimator to establish a faster rate.

Taking the class of all densities on the unit interval bounded from below by
\(1/r\) for some \(r > 1\) defines such an appropriately large class of
densities. That is to say, there cannot exist a consistent estimator and our
result gives a lower bound on the order of \(1/\sqrt{n}\). The experiment is
also within the purview of the result establishing our upper bound. For this
experiment, the rate of decay is therefore on the order of \(1/\sqrt{n}\).

To our knowledge, this is the first time a non-trivial rate other than \(1/n\) has been established. Since we are dealing with, in some sense, the «total information» in the experiment, one might have an intuition that a rate tending to zero should be linked to the existence of good estimator of the underlying measure. The above example contradicts this intuition, at least if good estimator is taken to mean an estimator consistent in total variation distance.

The structure of the paper is as follows. After settling on some notation in
Section~\ref{sec:notation}, the two main bounds and their application to the
example experiment discussed above are given in Section~\ref{sec:main}. In
Sections~\ref{sec:upperbound}~and~\ref{sec:lowerbound} the proofs of the upper
(Theorem~\ref{thm:mainupper}) and lower (Theorem~\ref{thm:mainlower}) bounds,
respectively, are given. Some technical proofs and results of an auxiliary
nature are postponed until Section~\ref{sec:proofs}.
Appendix~\ref{sec:basicfacts} recalls some basic facts necessary to follow the
proofs.

\section{Notation and Terminology}
\label{sec:notation}

This section outlines conventions for notation used throughout the paper. Basic
facts and definitions concerning these objects, that we may use without
mention, are collected in Appendix~\ref{sec:basicfacts}.

Throughout, let \(\N = \{1, \dotsc\}\), \(\Z\), and \(\R\) denote the (positive)
natural numbers, integers, and real numbers, respectively, with \(\R_{+} = \{ x
\in \R \mid x \geq 0\}\) the non-negative real numbers. For any natural number
\(n \in \N\) we will let \(\Delta_{n}\) denote the standard \((n-1)\)-simplex
given by \(\Delta_{n} = \{ (x_{1}, \dotsc, x_{n}) \in \R_{+}^{n} \mid x_{1} +
\dotsb + x_{n} = 1\}\). For any \(n \in \N\) and \((p_{1}, \dotsc, p_{n}) \in
\Delta_{n}\) we will have \(\PBin(p_{1}, \dotsc, p_{n})\) refer to the
Poisson-binomial distribution, that is to say the law of \(I_{1} + \dotsb +
I_{n}\) for \(I_{1}, \dotsc, I_{n}\) independent Bernoulli with \(\Prob(I_{i} =
1) = p_{i}\). Moreover, \(\Bin(n, p) = \PBin(p, \dotsc, p)\) for \(p \in [0,
1]\) and \(\Mult(n, p_{1}, \dotsc, p_{n})\) for \((p_{1}, \dotsc, p_{n}) \in
\Delta_{n}\) will denote, respectively, the Binomial and Multinomial
distributions.

Given a probability measure \(P\), we will use a slight abuse of notation by
writing the distribution function \(P((\infty, x])\) as \(\Prob(P \leq x)\) and
\(P((\infty, x))\) as \(\Prob(P < x)\). For example, \(\Prob(\PBin(p_{1}, \dotsc,
p_{k}) < l)\) is the probability of any random variable following the
Poisson-binomial law with parameters \(p_{1}, \dotsc, p_{k}\) being strictly
smaller than \(l\). Similarly, we will use the notation \(\Prob(P = l)\) and
even \(\Prob(P \in A)\) for the probabilities \(P(\{l\})\) and \(P(A)\), in case
it aids readability. That is to say, we will, for example, write \(\Prob(\Bin(n,p) = l)\) for the
probability of a Binomially distributed quantity with parameters \(n\) and \(p\)
being equal to \(l\).

An experiment \(\mathcal{E}\) on sample space \(X\), with parameter space
\(\Theta\), and \(\Theta\)-indexed family of measures \((P_{\theta})_{\theta \in
  \Theta}\) will be compactly written as \(\mathcal{E} = (X, P_{\theta} ; \theta
\in \Theta)\). The restriction of \(\mathcal{E}\) to some subset \(\Theta'
\subset \Theta\) is the experiment \((X, P_{\theta} ; \theta \in \Theta')\).

Whenever we say that a measurable space \(X = (\mathcal{X}, \mathcal{A})\) on
sample space \(\mathcal{X}\) with \(\sigma\)-algebra \(\mathcal{A}\) is
partitioned into (a measurable partition) \(X_{1}, \dotsc, X_{n}\) we mean that
\(X_{i} = (\mathcal{X}_{i}, \mathcal{A}_{i})\) are (measurable) subspaces of
\(X\) such that \(\bigcup_{i=1}^{n}\mathcal{X}_{i} = \mathcal{X}\).

When a measure \(\nu\) is absolutely continuous with respect to another measure
\(\nu\) we write \(\mu \ll \nu\). Writing \(d\mu = fd\nu\) for some
measurable \(f\) means \(f\) is a density (Radon-Nikodym derivative or
likelihood ratio) of \(\mu\) with respect to \(\nu\). The integral of some
\(\mu\)-integrable \(f\) is written \(\int f \,d\mu = \int f(x) \, \mu(dx)\).
\(\lambda\) will always denote the Lebesgue measure on the unit interval.

For any set \(A\), let \(\mathcal{D}(A) = (A, \mathcal{A})\) denote the
measurable space with underlying set \(A\) and the discrete \(\sigma\)-algebra
\(\mathcal{A} = 2^{A}\), the set of all subsets of \(A\). For a measurable space
\(X = (\mathcal{X}, \mathcal{A})\) we denote by \(\mathcal{P}(X)\) the
(measurable) space of probability measures on \(X\). The space
\(\mathcal{P}(X)\) is contained in a normed space of signed measures with the
norm given by the total variation norm. For \(P, Q \in \mathcal{P}(X)\) the
distance \(\|P - Q\|\) will always refer to the total variation distance.

Given a measurable \(f \colon X \to Y\), for two measurable spaces \(X\) and
\(Y\), we denote by \(f^{\ast}\) the corresponding (push-forward) map on
measures. In other words, 
\((f^{\ast}\mu)(A) = \mu(f^{-1}(A))\). The fact that \(K\) is a Markov-kernel from \(X\) to \(Y\) is written as \(K
\colon X \to_{k} Y\). We will not notationally distinguish \(K\) as a kernel
acting on \(X\) from \(K\) as a map acting on \(\mathcal{P}(X)\).

Products and direct sums of spaces, maps, kernels, measures, or experiments are
denoted by \(\cdot \otimes \cdot\) and \(\cdot \oplus \cdot\) or by
\(\bigotimes_{i \in I} \cdot_{i}\) and \(\bigoplus_{i \in I} \cdot_{i}\) for
\(I\)-indexed families of spaces (maps, kernels, measures, or experiments). By
analogy, we write \(\cdot^{\otimes n}\) for the \(n\):fold product of \(\cdot\).
For reasons of readability, we will implicitly use the fact that for any spaces
\(X\), \(Y\), and \(Z\) we have natural canonical (bimeasurable) isomorphisms
\((X \otimes Y) \otimes Z \cong X \otimes (Y \otimes Z) \cong X \otimes Y
\otimes Z\), where the lattermost is interpreted as a space of triplets. Any
reader concerned by this may wish to interpret expressions of the type
\(X^{\otimes n} \otimes Y^{\otimes m} \otimes Z^{\otimes r}\) as a shorthand for
the space of \((n+m+r)\)-tuples \(X \otimes \dotsb \otimes X \otimes Y \otimes
\dotsb \otimes Y \otimes Z \otimes \dotsb \otimes Z\), and similarly for
measures.

The mixture experiment (convex combination) of an \(I\)-indexed family of
experiments \((\mathcal{E}_{i})_{i \in I}\) with respect to convex coefficients
(probability mass function) \(p\) over \(I\) is denoted \(\sum_{i \in
  I}p_{i}\mathcal{E}_{i}\).

A decision \((A, L)\) problem on \(\Theta\) with action space \(A\) and loss
function \(L \colon A \times \Theta \to \R\) is said to be finite if \(A\) is
finite, normalised if the range is contained in \([0, 1]\), and 0-1 if the range
is contained in \(\{0, 1\}\). Saying \(\rho\) is a decision procedure for \((A,
L)\) on observing \(\mathcal{E} = (X, P_{\theta} ; \theta \in \Theta)\) means
\(\rho\) is a Markov kernel \(X \to_{k} A\). The risk \(R_{\mathcal{E}}(\rho,
\theta) = \int L(a, \theta) \, \rho_{x}(da) \, P_{\theta}(dx)\) will be
expressed in words as «the risk of \(\rho\) for \((L, A)\) at \(\theta \in
\Theta\) on observing \(\mathcal{E}\)». The average (Bayes) risk with respect to
some prior \(\pi\) on \(\Theta\) is expressed as «the risk of \(\rho\) for \((L,
A)\) with respect to prior \(\pi\) on observing \(\mathcal{E}\)». Similarly, the
minimum Bayes risk will be referred to as «the (minimum) Bayes risk of \((L,
A)\) with respect to prior \(\pi\) on observing \(\mathcal{E}\)» and any
procedure that achieves this risk is «Bayes for \((L, A)\) with respect to prior
\(\pi\) on observing \(\mathcal{E}\)». Any of \((L, A)\), \(\theta\), \(\pi\),
or \(\mathcal{E}\) that are understood from context may be suppressed.

The deficiency of an experiment \(\mathcal{E} = (X, P_{\theta}; \theta \in
\Theta)\) with respect to another experiment \(\mathcal{F}\) on the same
parameter space is denoted by \(\delta(\mathcal{E}, \mathcal{F})\). If
\(\delta(\mathcal{E}, \mathcal{F}) = 0\) then \(\mathcal{E}\) is said to be more
informative than \(\mathcal{F}\). If also \(\mathcal{F}\) is more informative
than \(\mathcal{E}\), they are said to be equivalent, denoted by \(\mathcal{E}
\cong \mathcal{F}\). A measurable map \(f \colon X \to Y\) such that
\(\mathcal{E}\) is equivalent to \(\mathcal{F} = (Y, f^{\ast}(P_{\theta}) ;
\theta \in \Theta)\) is said to be sufficient for \(\mathcal{E}\).

\section{Main results}
\label{sec:main}

Let \(\mathcal{E}\) be an experiment such that for each \(n \in \N\) we have an
\(n\):fold repetition \(\mathcal{E}^{\otimes n}\), the experiment given by
independently repeating \(\mathcal{E}\) independently \(n\) times. Our quantity
of interest is the \emph{value of an additional observation of \(\mathcal{E}\)
  when one already has \(n\)}, formalised as \(\delta(\mathcal{E}^{\otimes n},
\mathcal{E}^{\otimes n+1})\), the deficiency between the \(n\):fold and
\((n\!+\!1)\):fold repetition of \(\mathcal{E}\).

The main contribution of this paper is an upper and a lower bound for this
quantity for non-parametric experiments defined by sets of densities that are
very \emph{rich} in a sense that will be made precise.

Our upper bound holds for experiments that are contained in a certain type of
ball around a central probability measure.

\begin{samepage}
  \begin{definition}[\((C,s)\)-indistinguishably]
    For \(C \geq 1\), \(s > 0\) and measurable space \(X\), we will say that a
    probability measure \(Q\) on \(X\) is \((C,s)\)-indistinguishable from another
    probability measure \(P\) on \(X\) if there exists a density \(g\) on \(X\)
    such that \(dQ = gdP\) (a likelihood ratio) and for independent \(\xi_{1},
    \xi_{2}, \dotsc \sim P\) we have an exponential concentration of the form
    \begin{equation}\label{eq:indistinguishable}
      \Prob
      \Big(
      |\sum_{i=1}^{n}g(\xi_{i}) - n| > nt
      \Big) \leq Ce^{-snt^{2}}.
    \end{equation}
  \end{definition}
\end{samepage}

Taking \(n = 1\) in Equation~\eqref{eq:indistinguishable} means a necessary
condition is that \(g(\xi_{i})\) have a sub-Gaussian distribution. Conversely,
to establish Equation~\eqref{eq:indistinguishable} one may apply an appropriate
concentration inequality for sub-Gaussian random
variables\cite[Theorem~2.6.2]{Vershynin2018}. To satisfy
Equation~\eqref{eq:indistinguishable} for some \(C > 0\) and \(s > 0\) the
condition of sub-Guassianity is therefore both necessary and sufficient. The
definition is given in the above form in order to make the constants explicit in
a conceptually and notationally convenient way. One way the condition may be
satisfied for some \(C > 0\) and \(s > 0\) is for the likelihood ratio \(g\) to be
essentially bounded.

Conversely, it is immediately not satisfied if there exists some measurable
\(A\) such that \(Q(A) > 0\) but \(P(A) = 0\), since then \(\lim_{t \to
  \infty}\Prob(g(\xi_{1}) > t) \neq 0\). It is therefore necessary that \(Q\) be
absolutely continuous with respect to \(P\). This can be refined a bit further.
Observing that since, in Equation~\eqref{eq:indistinguishable}, \(g\) is the
likelihood ratio \(dQ/dP\) and \(\xi_{1} \sim P\) we have for any convex \(f\)
with \(f(1) = 0\) that \(\int f(g(x)) \, P(dx) = \Expect(f(g(\xi_{i})))\) is the
so called \(f\)-divergence \(D_{f}(Q, P)\). Taking \(f(x) = x \log(x)\) or
\(f(x) = (x - 1)^{2}\) yields, for example, the Kullback-Leibler and
\(\chi^{2}\)-divergences (see for example \cite[Section~III]{Liese2006}).

Taking \(t = 0\) in Equation~\ref{eq:indistinguishable} we see that,
necessarily, \(C \geq 1\). For \(n=1\) and \(t \geq \sqrt{\log(C)/s}\) the right
hand side of Equation~\ref{eq:indistinguishable} can be interpreted as the
survival function of a distribution with density \(x \mapsto 2sxCe^{-sx^{2}}\).
This can be interpreted as the distribution of \(|g(\xi_{1}) - 1|\) being
smaller than this distribution in (the usual) stochastic (dominance)
order\cite{Shaked2007}. Taking \(f(x) = (x - 1)^{2}\) this gives the following
bound on the \(\chi^{2}\)-divergence:
\begin{equation*}
  D_{f}(Q, P)
  =
  \Expect(f(g(\xi_{i})))
  =
  \Expect(|g(\xi_{i}) - 1|^{2})
  \leq
  \int_{\sqrt{\frac{\log(C)}{s}}}^{\infty} x^{2} 2sxCe^{-sx^{2}}\,dx.
\end{equation*}
The right-hand side is easily bounded in terms of the third moment of a Normal
distribution and depends only on \(C\) and \(s\). Thus the collection of \(Q\)
that are \((C,s)\)-indistinguishable from \(P\) are contained in some
\(\chi^{2}\)-ball with a radius controlled by \(C\) and \(s\). Similar, though
somewhat messier, arguments can be made for other divergences, such as Hellinger
distance or Kullback-Leibler divergence.

An experiment defined by measures that have a «central» measure
indistinguishable from all of them is, in some sense, not too large. Such
experiments exhibit a decay in the amount of information in additional
observations at a rate of at least \(1/\sqrt{n}\).

\begin{samepage}
  \begin{theorem}\label{thm:mainupper}
    Let \(\mathcal{E} = (X, P_{\theta} ; \theta \in \Theta)\) be an experiment
    such that for some \(C \geq 1\) and \(s > 0\) there exists a probability
    measure \(Q\) on \(X\) such that for each \(\theta \in \Theta\), \(Q\) is
    \((C,s)\)-indistinguishable from \(P_{\theta}\). Then
    \begin{equation*}
      \delta(\mathcal{E}^{\otimes n}, \mathcal{E}^{\otimes n+1})
      \leq
      C\sqrt{\frac{\pi}{4s}}\frac{1}{\sqrt{n+1}}.
    \end{equation*}
  \end{theorem}
\end{samepage}

This statement is essentially an application of Lemma~1 of
Helgeland\cite{Helgeland1982} and analogous to their Corollary~1 or to Theorem~1 of
Mammen\cite{Mammen1986}, but for a class of non-parametric experiments. The proof
is given in Section~\ref{sec:upperbound}.

A corresponding lower bound holds for experiments that are non-parametric in the
following sense. Observing \(\xi\) from the unknown, underlying measure \(P\) of
the experiment gives information about \(P\) only locally around \(\xi\). For
parametric experiments, by contrast, knowing some local structure of \(P\) may
well uniquely determine it. Making an analogy to
Devroye~\&~Lugosi\cite{Devroye2001} we say such families are \emph{rich}.

\begin{samepage}
  \begin{definition}[\((m,\alpha,\beta)\)-richness]
    For \(m \in \N\), \(\alpha \geq 0\), \(\beta \geq 0\), and a measurable space
    \(X\) we say that a family of probability measures \((P_{\theta})_{\theta \in
      \Theta}\) on \(X\) indexed by a set \(\Theta\) is \((m,\alpha,\beta)\)-rich if: there exists a partition
    \(\{X_{1}, \dotsc, X_{m}\}\) of \(X\); a convex coefficients \((p_{1}, \dotsc,
    p_{m}) \in \Delta_{m}\) such that \(p_{1}, \dotsc, p_{m} \geq \beta/m\); and
    pairs of distinct probability measures \((Q_{1,0},Q_{1,1}), \dotsc,
    (Q_{m,0},Q_{m,1})\) supported on \(X_{1}, \dotsc, X_{m}\), respectively, such
    that \(\| Q_{j,0} - Q_{j,1} \| \geq \alpha\)
    and for every \((i_{1}, \dotsc, i_{m}) \in \{0,1\}^{m}\)
    there exists \(\theta \in \Theta\) satisfying \(P_{\theta} = p_{1}Q_{1,i_{1}}
    + \dotsb + p_{m}Q_{m,i_{m}}\).
  \end{definition}
\end{samepage}

A typical example of such a family is given in the proof of
Corollary~\ref{cor:main} and can be seen in Figure~\ref{fig:diffononeparam}. The
choice of \(\alpha\) representing the total variation distance is for
concreteness. It holds more generally if replaced by \(\min(p,1-p) - r\) where
\(p \in (0, 1)\) and \(r\) is the Bayes risk of testing under a prior
probability putting mass \(p\) on \(Q_{j,0}\) and \(1-p\) on \(Q_{j,1}\). As
given here the definition is the special case where \(p = 1-p = 1/2\).

Taking \(\alpha > 0\) and \(\beta > 0\) typical parametric families
\((P_{\theta})_{\theta \in \Theta}\) will be \((m, \alpha, \beta)\)-rich only
for \(m = 1\). Consider for example, \(\Theta = \R\) with \(P_{\theta} =
\mathrm{N}(\theta, 1)\) the standard Gaussian shift family. Each \(P_{\theta}\)
is equivalent to the Lebesgue measure, meaning a set is a null set with respect
to some \(P_{\theta}\) if and only if it has Lebesgue measure zero. Let \(A\) be
some measurable set of positive Lebesgue measure. If \(P_{\theta}(B) =
P_{\theta'}(B)\) for each measurable \(B \subset A\), then \(\theta = \theta'\).
That is to say, \(\theta\) is uniquely determined by knowing \(P_{\theta}\)
restricted to any subspace of positive Lebesgue measure. This is most easily
seen by the fact that the standard, continuous, Gaussian density is uniquely
determined by knowing its value at three points and that the measure
\(P_{\theta}\) restricted to the subspace \(A\) has a density with respect to
Lebesgue measure that is essentially equal on \(A\) to exactly one such
continuous density.

Assume the family \((P_{\theta})_{\theta \in \Theta}\) were \((m, \alpha, \beta)\)-rich for some
\(\alpha > 0\) and \(\beta > 0\). The property is hereditary in the sense that
if it holds for \(m > 1\) it holds also for \(m-1\). If \(m \geq 2\) there would
therefore exist some partition \(X_{1}, X_{2}\) of \(\R\) where both \(X_{1}\)
and \(X_{2}\) have positive measure. Since \(P_{\theta}\) restricted to
\(X_{1}\) determines it on \(X_{2}\) there cannot exist measures \(Q_{1, 0},
Q_{1, 1}, Q_{2, 0}, Q_{2, 1}\) such that for some \(p \in (0, 1)\) both
\(pQ_{1,0} + (1-p)Q_{2,0}\) and \(pQ_{1,0} + (1-p)Q_{2,1}\) are Gaussian. The
family cannot therefore be \((m, \alpha, \beta)\)-rich for \(m > 1\).

Indeed, having a certain degree of richness is sufficient to get a lower bound
on the value of additional observations.

\begin{samepage}
  \begin{theorem}\label{thm:mainlower}
    Let \(\mathcal{E} = (X, P_{\theta} ; \theta \in \Theta)\) be an experiment
    such that for some \(n \in \N\), \(\alpha > 0\), and \(\beta > 0\) the family
    \((P_{\theta})_{\theta \in \Theta}\) is \((2n,\alpha,\beta)\)-rich, then %
    \begin{equation*}
      \frac{\alpha \beta}{12\sqrt{2}}
      \frac{1}{\sqrt{n+1}} \leq \delta(\mathcal{E}^{\otimes n}, \mathcal{E}^{\otimes n+1}).
    \end{equation*}
  \end{theorem}
\end{samepage}

This result is somewhat comparable to Proposition~2 of
Helgeland\cite{Helgeland1982} and Theorem~2 of Mammen\cite{Mammen1986}, both of
which concern experiments that are finite-dimensional in an appropriate sense.
The proof is given in Section~\ref{sec:lowerbound}.

In particular, if for some \(\alpha > 0\) and \(\beta > 0\) a family is \((m,
\alpha, \beta)\)-rich for all \(m \in \N\) then the value of additional
observations from the corresponding experiment decreases at a rate of
\(1/\sqrt{n}\).

To construct an example of such an experiment, let \(X\) be the unit interval and consider the experiment
\begin{equation}\label{eq:mainexperiment}
  \mathcal{E} = (X, P_{f} ; f \in \Theta) 
  \quad
  \Theta = \{f : X \to \R \mid 1/r \leq f, \lambda(f) = 1\}
\end{equation}
for some \(r > 1\) and where \(dP_{f} = fd\lambda\). The choice of the unit
interval with Lebesgue/uniform dominating measure is mostly for concreteness.

The experiment in Equation~\eqref{eq:mainexperiment} is interesting, on the one
hand, for being an experiment for which \(\delta(\mathcal{E}^{\otimes n},
\mathcal{E}^{\otimes n+1})\) will turn out to decay at a non-trivial rate other
than the parametric one of \(1/n\). On the other hand, it is interesting also
because it is impossible to consistently estimate \(P_{f}\) in total variation
distance (see Theorem~15.1 and succeeding remarks in Devroye\cite{Devroye2001}).
In other words, one is in some sense exhausting information, but doing so without being
able to precisely determine the underlying distribution. Note that these
statements are not in contradiction, since the deficiency distance takes into
account only finite decision problems.

For the experiment described in Equation~\ref{eq:mainexperiment} we may use
Theorems~\ref{thm:mainupper}~and~\ref{thm:mainlower} to identify the rate at which
the value of having one additional observation decays.

\begin{samepage}
  \begin{corollary}\label{cor:main}
    For the experiment \(\mathcal{E}\) in Equation~\eqref{eq:mainexperiment} we
    have
    \begin{equation*}
      c\frac{1}{\sqrt{n+1}}
      \leq
      \delta(\mathcal{E}^{\otimes n}, \mathcal{E}^{\otimes n+1})
      \leq
      C\frac{1}{\sqrt{n+1}},
    \end{equation*}
    where \(c = \frac{1 - 1/r}{12\sqrt{2}}\) and \(C = \sqrt{\frac{\pi}{2}}r < 1.3r\).
  \end{corollary}
\end{samepage}

\begin{proof}
  To get the upper bound, apply Theorem~\ref{thm:mainupper} with \(Q = \lambda\)
  the uniform (Lebesgue) measure on \([0, 1]\). By Hoeffding's inequality for
  bounded distributions we know that \(\lambda\) is
  \((2,2/r^{2})\)-indistinguishable from \(P_{f}\) for each \(f \in \Theta\).
  Thus, by Theorem~\ref{thm:mainupper}, we know \(\delta(\mathcal{E}^{\otimes n},
  \mathcal{E}^{\otimes n+1}) \leq \sqrt{\frac{\pi}{2}}r\frac{1}{\sqrt{n+1}}\).

  In order to establish the lower bound using Theorem~\ref{thm:mainlower}, we
  prove that the family \((P_{f})_{f \in \Theta}\) is \((m, 1-1/r, 1)\)-rich for
  every \(m \in \N\).

  Take the partition given by the regular mesh \(X_{1}, \dotsc, X_{m-1}, X_{m}\)
  with underlying sets \([0, 1/m), \dotsc, [1-2/m, 1-1/m), [1-1/m, 1]\), uniform
  convex coefficients \(p_{1} = 1/m, \dotsc, p_{m} = 1/m\), and pairs of
  distributions \((Q_{1,0}, Q_{1,1}), \dotsc, (Q_{m,0}, Q_{m,1})\) specified by
  \(dQ_{i,j} = g_{i,j}d\lambda\) for \(g_{i,0}(x) = mg_{0}((i-1+x)/m)\) and
  \(g_{i,1} = mg_{1}((i-1+x)/m)\) where
  \begin{align}\label{eq:simpledensities}
    &g_{0}(x) = 
    \begin{cases}
      1/r & x \leq 1/2 \\
      2 - 1/r & x > 1/2
    \end{cases}
    &&\text{and}
    &&g_{1}(x) =
    \begin{cases}
      2 - 1/r & x \leq 1/2 \\
      1/r & x > 1/2
    \end{cases}.
  \end{align}
  For any \((i_{1}, \dotsc, i_{m}) \in \{0, 1\}^{m}\) consider
  \begin{equation*}
    f(x) =
    \begin{cases}
      g_{i_{1}}(x/m) & x \in X_{1} \\
      \vdotswithin{g_{i_{m}}} \\
      g_{i_{m}}((m-1+x)/m) & x \in X_{m}
    \end{cases},
  \end{equation*}
  such that we have \(P_{f} = \sum_{j=1}^{m}p_{j}P_{j,i_{j}} = fd\lambda\). The
  densities are illustrated in Figure~\ref{fig:diffononeparam}.
  
  Since \(f \in \Theta\) it remains only to realise that for each \(j \in \{1,
  \dotsc, m\}\) the total variation distance between \(Q_{j,0}\) and \(Q_{j,1}\)
  is \(1-1/r\). \qedhere{}
\end{proof}

\begin{figure}[ht]
  \centering%
  \includegraphics{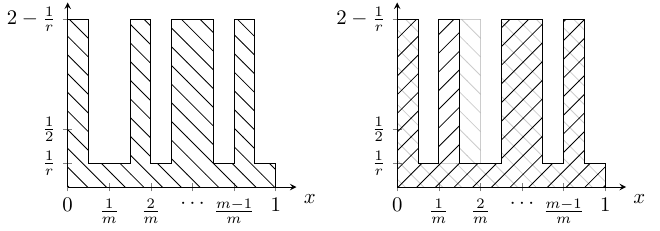}
  \caption{Plots of \(f\) in the proof of Corollary~\ref{cor:main} differing
    only on \(i_{2}\)}
  \label{fig:diffononeparam}
\end{figure}

Note that the assumption that all densities are bounded away from \(0\) by some
constant is fairly common. It appears, for example, in the classical result of
Nussbaum\cite{Nussbaum1996} on the asymptotic equivalence to a white noise model
when observing an unknown density from a certain smoothness class.

\section{Proof of upper bound (Theorem~\ref{thm:mainupper})}
\label{sec:upperbound}

The aim of this section is to prove Theorem~\ref{thm:mainupper}, for which we
use a technique due to Helgeland\cite{Helgeland1982}. The main idea is to
emulate \(n+1\) independent observations based on \(n\) independent observations
by injecting, in a random position, a new, randomised, value. In Helgeland's
proof the value is sampled from an estimate of the underlying distribution,
based on the \(n\) truly independent and identically distributed observations.
In our proof this estimated distribution will be replaced by a single, fixed
distribution. To improve readability, we repeat the relevant parts from the
proof of Helgeland's Lemma~1.

\begin{proof}[Proof of Theorem~\ref{thm:mainupper}]  
  Let \(C \geq 1\), \(s > 0\), \(\mathcal{E} = (X, P_{\theta} ; \theta \in
  \Theta)\) and \(Q\) be as in the statement, such that \(dQ =
  g_{\theta}P_{\theta}\) for a family of densities \((g_{\theta})_{\theta \in \Theta}\). By assumption
  these satisfy that for each \(\theta \in \Theta\) and for \(\xi_{1}, \xi_{2},
  \dotsc \sim P_{\theta}\) that
  \begin{equation*}
    \Prob
    \Big(
    |\sum_{i=1}^{n}g_{\theta}(\xi_{i}) - n| > nt
    \Big) \leq Ce^{-snt^{2}}.
  \end{equation*}

  Recall that the standard technique for bounding deficiencies from above
  follows from the fact that for any Markov-kernel \(K : X^{\otimes n} \to_{k}
  X^{\otimes n+1}\) we have
  \begin{equation*}
    \delta(\mathcal{E}^{\otimes n}, \mathcal{E}^{\otimes n + 1})
    \leq
    \sup_{\theta}\| KP_{\theta}^{\otimes n} -  P_{\theta}^{\otimes n+1}\|.
  \end{equation*}
  This is either immediate from the definition (for example in
  Le~Cam\cite[Definition~2.3.1]{LeCam1986}) or a Theorem (for example in
  Torgersen\cite[Theorem~6.2.4]{Torgersen1991}).
  Take as \(K\) a kernel that injects, in a random position, an additional
  randomised observation from \(Q\). The idea is illustrated in Figure~\ref{fig:kernelpicture}.
  \begin{figure}[ht]
    \centering%
    \includegraphics[trim={0 0 0 0.75cm}]{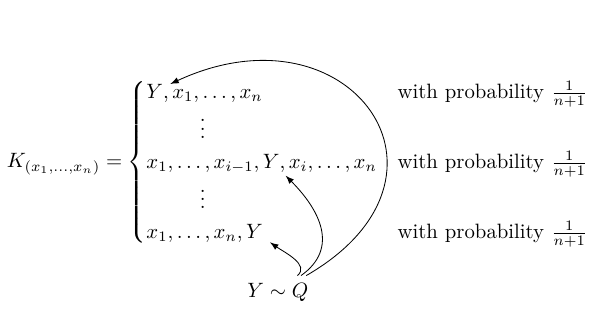}%
    \caption{Illustration of kernel \(K\) in proof of Theorem~\ref{thm:mainupper}}
    \label{fig:kernelpicture}
  \end{figure}

  Formally for each \(x = (x_{1}, \dotsc, x_{n}) \in X^{\otimes n}\) define
  \(K\) by
  \begin{equation*}
    K_{x}(A) = \frac{1}{n+1}\sum_{i=1}^{n+1}\int \mathbf{1}_{A}(x_{1}, \dotsc, x_{i-1}, y, x_{i}, \dotsc, x_{n})\,Q(dy),
  \end{equation*}
  such that In particular,
  \begin{equation*}
    KP_{\theta}^{\otimes n} = \frac{1}{n+1}\sum_{i=1}^{n+1}P_{\theta}^{\otimes i-1} \otimes Q \otimes P_{\theta}^{\otimes n-i+1}.
  \end{equation*}

  Using this kernel, bounding the deficiency turns into a question of
  controlling the absolute deviation of an average.

  Fix some \(\theta \in \Theta\) and let \(\xi_{1}, \dotsc, \xi_{n+1} \sim
  P_{\theta}\) be independent. Rewriting the total variation distance between
  \(KP_{\theta}^{\otimes n}\) and \(P_{\theta}^{\otimes n+1}\) in terms of the
  \(L^{1}\)-distance between their densities with respect to
  \(P_{\theta}^{\otimes n+1}\) yields
  \begin{equation*}
    \begin{split}
      2\| KP_{\theta}^{\otimes n} -  P_{\theta}^{\otimes n+1}\| &=
      2\Big\| \frac{1}{n+1}\sum_{i=1}^{n+1} P_{f}^{\otimes i-1} \otimes Q \otimes P_{f}^{\otimes n-i+1} - P_{f}^{\otimes n+1} \Big\|
      \\&=
      \int \Big|\frac{1}{n+1}\sum_{i=1}^{n+1}g_{\theta}(x_{i}) - 1\Big|\,P_{\theta}^{\otimes n+1}(dx)
      \\&=
      \Expect\Big(\Big|\frac{1}{n+1}\sum_{i=1}^{n+1}g_{\theta}(\xi_{i}) - 1\Big|\Big).
    \end{split}
  \end{equation*}
  By assumption \(g_{\theta}(\xi_{1}), \dotsc, g_{\theta}(\xi_{n+1})\) are
  independent and identically distributed. Moreover, \(g_{\theta}\) is the
  density of \(Q\) with respect to \(P_{\theta}\) so that
  \(\Expect(g_{\theta}(\xi_{i})) = 1\). We are thus reduced to bounding the mean
  absolute deviation of the average
  \(\frac{1}{n+1}\sum_{i=1}^{n+1}g_{\theta}(\xi_{i})\) from its mean \(1\). Such
  a bound follows immediately from the assumption that \(Q\) is
  \((C,s)\)-indistinguishable from \(P_{\theta}\) since then
  \begin{equation*}
    \begin{split}
      \Expect\Big(\Big|\frac{1}{n+1}\sum_{i=1}^{n+1}g_{\theta}(\xi_{i}) - 1\Big|\Big)
      &=
      \int_{0}^{\infty} \Prob\Big(\Big|\frac{1}{n+1}\sum_{i=1}^{n+1}g_{\theta}(\xi_{i}) - 1\Big| > t\Big) \,dt
      \\&\leq
      C\sqrt{\frac{\pi}{(n+1)s}}\frac{1}{2} \int_{-\infty}^{\infty} \frac{1}{\sqrt{\frac{\pi}{(n+1)s}}} e^{-(n+1)st^{2}}\,dt
      \\&=
      C\sqrt{\frac{\pi}{4(n+1)s}},
    \end{split}
  \end{equation*}
  where the inequality is exactly the concentration inequality in the definition
  of \(Q\) being \((C,s)\)-indistinguishable from \(P_{\theta}\). Combining the
  above two inequalities, we get
  \begin{equation*}
    \delta(\mathcal{E}^{\otimes n}, \mathcal{E}^{\otimes n + 1})
    \leq
    \sup_{\theta} C\sqrt{\frac{\pi}{4(n+1)s}}
    =
    C\sqrt{\frac{\pi}{4s}}\frac{1}{\sqrt{n+1}}.\qedhere
  \end{equation*}
\end{proof}

With the method used here it would be impossible to establish a faster rate of
decay with respect to \(n\). This follows from a result by
Mattner\cite{Mattner2003} which implies that
\begin{equation*}
  \Expect\Big(\Big|\sum_{i=1}^{n+1}g_{\theta}(\xi_{i}) - 1\Big|\Big)
  \geq
  \frac{1}{\sqrt{2(n+1)}}\Expect(|g_{\theta}(\xi_{1}) - 1|).
\end{equation*}
In other words, the absolute mean deviation can, if finite, not decay more
quickly than at a rate of \(1/\sqrt{n}\).

\section{Proof of lower bound (Theorem~\ref{thm:mainlower})}
\label{sec:lowerbound}

Establishing the lower bound in Theorem~\ref{thm:mainlower} is a bit more
involved than was the case for the upper bound. The proof may be found at the end of
this section. For clarity and readability, the proofs of some intermediate lemmas
and technical results are postponed until Section~\ref{sec:proofs}.

Similarly,  to certain techniques for establishing minimax bounds we will rely on
the existence of appropriate hypercubes of parameters. These hypercubes will
yield multiple testing problems that become significantly easier with additional
observations and thus give a lower bound on the deficiency of interest.

These multiple testing problems can be thought of as consisting of many local hypotheses about the
underlying distribution. An additional observation will allow one to make an
informed guess about the shape of the underlying distribution in an additional
small region of the sample space. The idea is illustrated in
Figure~\ref{fig:lowerboundidea}.

\begin{figure}[ht]
  \centering
  \includegraphics{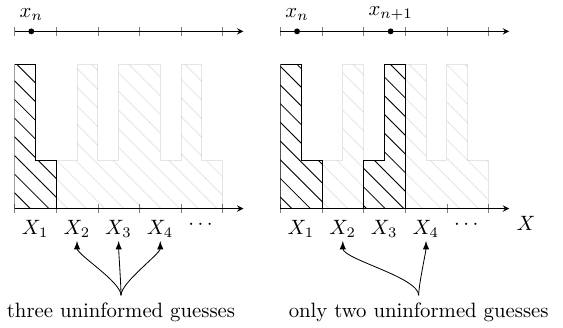}
  \caption{Illustration of the proof of Theorem~\ref{thm:mainlower}, observing
    \(x_{n+1}\) typically makes it viable to estimate the density in one
    additional cell of a partition \(X_{1}, \dotsc, X_{m}\) of the sample space
    \(X\).}
  \label{fig:lowerboundidea}
\end{figure}

The relevant notion of locality is captured by re-parameterising the experiment in
an appropriate way, as described by the following lemma.

\begin{samepage}
\begin{lemma}\label{lemma:hypercube}

  Let \(\mathcal{E} = (X, P_{\theta}; \theta \in \Theta)\) be an experiment on a space \(X\), and
  for some \(m \in \N\) let \(\{X_{1}, \dotsc, X_{m}\}\) be a partition of \(X\)
  such that for all \(i = 1, \dotsc, m\) and \(\theta \in \Theta\) one has
  \(P_{\theta}(X_{i}) > 0\). Also define for each \(i = 1, \dotsc, m\) and
  \(\theta \in \Theta\)
  \begin{equation*}
    P_{i,\theta}(A) = P_{\theta}(A \mid X_{i}) \qquad \Theta_{i} = \{ P_{i,\theta} \mid \theta \in \Theta \}
  \end{equation*}
  as well as \(\psi \colon \Theta \to \Delta_{m} \times \Theta_{1} \times
  \dotsb \times \Theta_{m}\) by
  \begin{equation*}
    \psi(\theta)
    =
    ((P_{\theta}(X_{1}), \dotsc, P_{\theta}(X_{m})), P_{1,\theta}, \dotsc, P_{m,\theta})
  \end{equation*}
  and let \(\Psi = \psi(\Theta)\). Then the experiment
  \begin{equation}\label{eq:splitexperiment}
    \mathcal{F} = 
    \Bigl(\bigoplus_{i=1}^{m}X_{i}, \bigoplus_{i=1}^{m}p_{i}P_{i} ; ((p_{1}, \dotsc, p_{m}), P_{1}, \dotsc, P_{m}) \in \Psi\Bigr)
  \end{equation}
  is such that for any pair \(n,n' \in \N\) one has
  \begin{equation*}
    \delta(\mathcal{E}^{\otimes n}, \mathcal{E}^{\otimes n'})
    =
    \delta(\mathcal{F}^{\otimes n}, \mathcal{F}^{\otimes n'}).
  \end{equation*}
\end{lemma}
\end{samepage}

The proof is theoretically trivial but in practice somewhat technical and can be
found in Section~\ref{sec:proofs}. One may think of the new parameterisation in terms of \(\Psi \subset \Delta_{m} \times \Theta_{1}
\times \dotsb \times \Theta_{n}\) specified by \(\psi\) as essentially
nothing but the law of total probability. It makes explicit how each
\(P_{\theta}\) decomposes into the coarse structure given by
\((P_{\theta}(X_{1}), \dotsc, P_{\theta}(X_{m})) \in \Delta_{m}\) and the local
pieces \(P_{\theta_{1}}, \dotsc, P_{\theta_{m}}\) within each cell \(X_{1},
\dotsc, X_{m}\) of the partition.

In particular, if \(p = (p_{1}, \dotsc, p_{m}) \in \Delta_{m}\) is fixed the
experiment in Equation~(\ref{eq:splitexperiment}) is a mixture \(\mathcal{F} =
\sum_{i=1}^{m}p_{i}\mathcal{F}_{i}\) where \(\mathcal{F}_{i} = (X_{i},
P_{i,\theta_{i}}; (\dotsc, \theta_{i}, \dotsc) \in \Psi)\) depends only on
\(\theta_{i}\). For such a mixture \(\mathcal{F}\), receiving a number of
observations can be thought of as observing a randomised smaller number of
observations for each one of \(\mathcal{F}_{1}, \dotsc, \mathcal{F}_{m}\). This
notion is formalised by the following lemma, which is nothing but a multinomial
theorem for experiments.

\begin{lemma}\label{lemma:mixedpower}
  For \(n, m \in \N\) let
  \begin{equation*}
    \mathcal{E}_{1} = (X_{1}, P_{1,\theta} ; \theta \in
    \Theta), \dotsc, \mathcal{E}_{m} = (X_{m}, P_{m,\theta} ; \theta \in \Theta)
  \end{equation*}
  be experiments on the same parameter space \(\Theta\), \(p = (p_{1}, \dotsc,
  p_{m}) \in \Delta_{m}\), and denote \(q(n_{1}, \dotsc, n_{m}) = \Prob(\Mult(n,
  p_{1}, \dotsc, p_{m}) = (n_{1},\dotsc,n_{m}))\). Then
  \begin{equation*}
    \Bigl(
    \sum_{i=1}^{m}p_{i}\mathcal{E}_{i}
    \Bigr)^{\otimes n}
    \cong
    \sum_{n_{1}, \dotsc, n_{m}} q(n_{1}, \dotsc, n_{m}) \, \mathcal{E}_{1}^{\otimes n_{1}} \otimes \dotsb \otimes \mathcal{E}_{m}^{\otimes n_{m}}.
  \end{equation*}
\end{lemma}
The proof is uninteresting but requires a bit of bookkeeping, and is therefore
postponed until Section~\ref{sec:proofs}.

For such experiments we may derive a lower bound in terms of the difficulty of
testing problems in the individual experiments.

\begin{lemma}\label{lemma:multitestbound}
  Fix some \(m \in \N\), parameter space \(\Theta \subset \Theta_{1} \times
  \dotsb \times \Theta_{m}\), a sequence of experiments
  \begin{equation*}
    \begin{split}
      \makebox[\widthof{\(\mathcal{E}_{m}\)}][c]{\(\mathcal{E}_{1}\)} &= (\makebox[\widthof{\(X_{m}\)}][c]{\(X_{1}\)}, \makebox[\widthof{\(P_{m,\theta_{m}}\)}][c]{\(P_{1,\theta_{1}}\)} ; (\theta_{1}, \dotsc, \theta_{m}) \in \Theta),\\
      &\vdotswithin{=}\\
      \mathcal{E}_{m} &= (X_{m}, P_{m,\theta_{m}} ; (\theta_{1}, \dotsc, \theta_{m}) \in \Theta),
    \end{split}
  \end{equation*}
  where each family \((P_{j,\theta_{j}})_{\theta \in \Theta}\) depends only on
  \(\theta_{j}\), and there exist pairs \(\theta_{1,0}, \theta_{1,1} \in \Theta_{1}\),
  \ldots{}, \(\theta_{m,0}, \theta_{m,1} \in \Theta_{m}\) such that
  \(\{\theta_{1,0}, \theta_{1,1}\} \times \dotsb \times \{\theta_{m,0},
  \theta_{m,1}\} \subset \Theta\). For \(q, q' \colon \N^{m} \to [0, 1]\) two
  probability mass functions on \(\N^{m}\) define experiments
  \begin{equation*}
    \begin{split}
      \makebox[\widthof{\(\mathcal{E}'\)}][c]{\(\mathcal{E}\)} &=
      \sum_{n_{1}, \dotsc, n_{m}} \makebox[\widthof{\(q'(n_{1}, \dotsc, n_{m})\)}][c]{\(q(n_{1}, \dotsc, n_{m})\)}
      \,\mathcal{E}_{1}^{\otimes n_{1}} \otimes \dotsb \otimes \mathcal{E}_{m}^{\otimes n_{m}},
      \\
      \mathcal{E}' &=
      \sum_{n_{1}, \dotsc, n_{m}} q'(n_{1}, \dotsc, n_{m})
      \,\mathcal{E}_{1}^{\otimes n_{1}} \otimes \dotsb \otimes \mathcal{E}_{m}^{\otimes n_{m}}.
    \end{split}
  \end{equation*}

  Fixing some priors \(\pi_{1}, \dotsc, \pi_{m}\) supported on the pairs of
  parameters \(\{\theta_{1,0}, \theta_{1,1}\}, \dotsc, \{\theta_{m,0},
  \theta_{m,1}\}\) let \(r_{i}(n)\) denote the minimum Bayes risks of testing
  \(P_{i,\theta_{i,0}}^{\otimes n}\) against \(P_{i,\theta_{i,1}}^{\otimes n}\)
  with respect to \(\pi_{i}\).
  
  If \((N_{1}, \dotsc, N_{m})\) and \((N'_{1}, \dotsc, N'_{m})\) are distributed
  according to \(q\) and \(\mathrlap{q'}\phantom{q}\), respectively, it then
  holds for any \(l=1, \dotsc, m\) that
  \begin{equation}\label{eq:cubelower}
    \begin{split}
      \delta(\mathcal{E}, \mathcal{E}')
      \geq\,&
      \Expect \Bigl( 
      \Prob(\PBin(r_{1}(N_{1}), \dotsc, r_{m}(N_{m})) \geq l)
      \Bigr)
      \\&-
      \Expect \Bigl( 
      \Prob(\PBin(r_{1}(N'_{1}), \dotsc, r_{m}(N'_{m})) \geq l)
      \Bigr).
    \end{split}
  \end{equation}
\end{lemma}
Due to its relative length, the proof is found in Section~\ref{sec:proofs}.
Conceptually, it captures exactly what was illustrated in
Figure~\ref{fig:lowerboundidea}, with each parameter \(r_{i}(N_{i})\)
corresponding to the probability of making an incorrect decision about the
underlying distribution within the cell \(X_{i}\) of a partition \(\{X_{1},
\dotsc, X_{m}\}\) of the sample space \(X\). If \(q'\) describes a distribution that is in
some sense larger than the one described by \(q\) then \(N'_{i}\) will tend to
be greater than \(N_{i}\) and the risk \(r_{i}(N'_{i})\) will tend to be
smaller.

In our case, controlling the right-hand side in Equation~\ref{eq:cubelower}
will boil down to proving that such mixtures of Poisson-Binomial distributions
are concentrated enough for there to exist outcomes with probabilities at least
on the order of \(1/\sqrt{m}\). For this, we need the following two
simple properties.
\begin{lemma}\label{lemma:pbindiff}

  Let \(n \in \N\), \(i = 1,\dotsc,n\), and \(p_{1}, \dotsc, p_{i}, p_{i}',
  \dotsc, p_{n+1} \in [0,1]\) where \(p_{i} > p'_{i}\). Then for any \(l = 1,
  \dotsc, n\)
  \begin{equation*}
    \begin{split}
      &\Prob(\PBin(p_{1}, \dotsc, p_{i}, \dotsc, p_{n+1}) \geq l)
      -
      \Prob(\PBin(p_{1}, \dotsc, p'_{i}, \dotsc, p_{n+1}) \geq l)
      =
      \\
      &\qquad=
      \Prob(\PBin(p_{1}, \dotsc, p_{i-1}, p_{i+1}, \dotsc, p_{n+1})
      =
      (l-1))(p_{i}-p'_{i}).
    \end{split}
  \end{equation*}
\end{lemma}
The proof is a simple calculation and can be found in Section~\ref{sec:proofs}.

\begin{lemma}\label{claim:mixedpbin}
  For any \(m \in \N\), \(n > 0\), \((p_{1}, \dotsc, p_{m}) \in \Delta_{m}\),
  and any sequence of (all non-decreasing or all non-increasing) monotone functions \(f_{1}, \dotsc, f_{m} : \N \to [0, 1]\) there
  exists a \(k \in \{0, 1, \dotsc, m\}\) such that for \((N_{1}, \dotsc, N_{m})
  \sim \Mult(n, p_{1}, \dotsc, p_{m})\) one has
  \begin{equation*}
    \Expect(\Prob(\PBin(f_{1}(N_{1}), \dotsc, f_{m}(N_{m})) = k)) \geq \frac{1}{3\sqrt{m}}.
  \end{equation*}
\end{lemma}
Also this proof can be found in Section~\ref{sec:proofs}. We now have all the pieces necessary to prove our lower bound.

\begin{proof}[Proof of Theorem~\ref{thm:mainlower}]
  Let \(n \in \N\), \(\alpha > 0\), and \(\mathcal{E} = (X, P_{\theta} ; \theta
  \in \Theta)\) be as in the statement. By assumption there exists a partition
  \(X_{1}, \dotsc, X_{2n}\) of \(X\); \(p = (p_{1}, \dotsc, p_{2n}) \in
  \Delta_{2n}\) with \(p_{1}, \dotsc, p_{2n} > \beta/(2n)\); and probability
  measures \((Q_{1,0}, Q_{1,1}), \dotsc, (Q_{2n,0}, Q_{2n, 1})\) supported on
  \(X_{1}, \dotsc, X_{2n}\) such that the total variation distances
  \(\|Q_{1,0}-Q_{1,1}\|, \dotsc, \|Q_{2n,0}-Q_{2n,1}\|\) are all at least
  \(\alpha\). Moreover, for each \((i_{1}, \dotsc, i_{2n}) \in \{0, 1\}^{2n}\)
  there exists a \(\theta \in \Theta\) such that \(P_{\theta} = p_{1}Q_{1,i_{1}}
  + \dotsb + p_{2n}Q_{2n,i_{2n}}\).

  By Lemma~\ref{lemma:hypercube} we have \(\delta(\mathcal{E}^{\otimes n},
  \mathcal{E}^{\otimes n+1}) = \delta(\mathcal{F}^{\otimes n},
  \mathcal{F}^{\otimes n+1})\), where \(\mathcal{F}\) is as in
  Equation~\ref{eq:splitexperiment} with respect to the partition \(X_{1},
  \dotsc, X_{2n}\). Let \(\mathcal{F}_{p}\) be the restriction of
  \(\mathcal{F}\) to \(\Psi' = \{ (p', P_{1}, \dotsc, P_{n}) \in \Psi \mid p' =
  p \} \subset \Psi\). Since \(\mathcal{F}_{p}\) is a restriction of \(\mathcal{F}\) it follows that \(\delta(\mathcal{F}_{p}^{\otimes n},
  \mathcal{F}_{p}^{\otimes n+1}) \leq \delta(\mathcal{F}^{\otimes n},
  \mathcal{F}^{\otimes n+1})\).

  By definition \(\mathcal{F}_{p} = \sum_{i=1}^{2n} p_{i} \mathcal{F}_{p,i}\)
  where
  \begin{equation*}
    \begin{split}
      \makebox[\widthof{\(\mathcal{F}_{p,2n}\)}][c]{\(\mathcal{F}_{p,1}\)} &= (\makebox[\widthof{\(X_{2n}\)}][c]{\(X_{1}\)}, \makebox[\widthof{\(P_{2n}\)}][c]{\(P_{1}\)} ; (p, P_{1}, \dotsc, P_{2n}) \in
      \Psi')
      \\
      &\vdotswithin{=}
      \\
      \mathcal{F}_{p,2n} &= (X_{2n}, P_{2n} ; (p, P_{1}, \dotsc,
      P_{2n}) \in \Psi').
    \end{split}
  \end{equation*}
  Using Lemma~\ref{lemma:mixedpower} this implies
  \(\delta(\mathcal{F}_{p}^{\otimes n}, \mathcal{F}_{p}^{\otimes n}) =
  \delta(\mathcal{G}, \mathcal{G}')\) where for
  \begin{equation*}
    \begin{split}
      \mathrlap{q}\phantom{q'}(n_{1}, \dotsc, n_{2n}) &= 
      \Prob(\Mult(\makebox[\widthof{\(n+1\)}][c]{\(n\)}, p_{1}, \dotsc, p_{2n}) = (n_{1},\dotsc,n_{2n})),
      \\
      q'(n_{1}, \dotsc, n_{2n}) &=
      \Prob(\Mult(n+1, p_{1}, \dotsc, p_{2n}) = (n_{1},\dotsc,n_{2n})),
    \end{split}
  \end{equation*}
  \(\mathcal{G}\) and \(\mathcal{G}'\) are given by
  \begin{equation*}
    \begin{split}
      \mathrlap{\mathcal{G}}\phantom{\mathcal{G}'} &= 
      \sum_{n_{1}, \dotsc, n_{2n}} \mathrlap{q}\phantom{q'}(n_{1},\dotsc,n_{2n})
      \,\mathcal{F}_{p,1}^{\otimes n_{1}} \otimes \dotsb \otimes \mathcal{F}_{p,2n}^{\otimes n_{2n}},
      \\
      \mathcal{G}' &=
      \sum_{n_{1}, \dotsc, n_{2n}} q'(n_{1},\dotsc,n_{2n})
      \,\mathcal{F}_{p,1}^{\otimes n_{1}} \otimes \dotsb \otimes \mathcal{F}_{p,2n}^{\otimes n_{2n}}.
    \end{split}
  \end{equation*}
  Finally Lemma~\ref{lemma:multitestbound} lets us reduce the
  problem to one of basic probability. Let \(r_{1}, \dotsc, r_{n} \colon \N \to
  [0, 1]\) be the Bayes risks as in Lemma~\ref{lemma:multitestbound} with
  respect to uniform priors. By assumption \(2r_{i}(1) = 1 - \|P_{i,\theta_{i}}
  - P_{i,\theta'_{i}}\|\). Lemma~\ref{lemma:multitestbound} lets us conclude that for any \(l = 0, \dotsc,
  2n-1\), \(N = (N_{1}, \dotsc, N_{2n}) \sim \Mult(n, p_{1}, \dotsc, p_{2n})\),
  and \(N' = (N'_{1}, \dotsc, N'_{2n}) \sim \Mult(n+1, p_{1}, \dotsc, p_{2n})\)
  \begin{equation}\label{eq:mainlowerprob}
    \begin{split}
      \delta(\mathcal{E}^{\otimes n}, \mathcal{E}^{\otimes n+1})
      \geq&
      \Expect \Bigl( 
      \Prob(\PBin(r_{1}(N_{1}), \dotsc, r_{2n}(N_{2n})) \geq l+1)
      \Bigr)
      \\&- 
      \Expect \Bigl( 
      \Prob(\PBin(r_{1}(N'_{1}), \dotsc, r_{2n}(N'_{2n})) \geq l+1)
      \Bigr).
    \end{split}
  \end{equation}

  Let \(I = (I_{1}, \dotsc, I_{2n}) \sim \Mult(1, p_{1}, \dotsc, p_{2n})\) be
  independent of \(N\) and introduce the notations \(R_{i} = r_{i}(N_{i})\) and
  \(R'_{i} = r_{i}(N_{i} + 1)\). The random vector \((N_{1} + I_{1}, \dotsc,
  N_{2n} + I_{2n})\) has the same distribution as \(N'\), so
  \begin{equation*}
    \begin{split}
      &\Expect ( 
      \Prob(\PBin(r_{1}(N'_{1}), \dotsc, r_{2n}(N'_{2n})) \geq l+1)
      )
      =
      \\
      &\qquad= \Expect ( 
      \Prob(\PBin(r_{1}(N_{1}+I_{1}), \dotsc, r_{2n}(N_{2n}+I_{2n})) \geq l+1)
      )
      \\
      &\qquad= \sum_{i=1}^{2n}p_{i}
      \Expect ( 
      \Prob(\PBin(R_{1}, \dotsc, R_{i-1}, R'_{i}, R_{i+1} \dotsc, R_{2n}) \geq l+1)
      ).
    \end{split}
  \end{equation*}

  Plugging the above into Equation~\ref{eq:mainlowerprob} and using
  Lemma~\ref{lemma:pbindiff} gives a lower bound of
  \begin{equation*}
    \sum_{i=1}^{2n}p_{i}
    \Expect ( 
    (R_{i}-R'_{i})
    \Prob(\PBin(R_{1}, \dotsc, R_{i-1}, R_{i+1}, \dotsc, R_{2n}) = l)).
  \end{equation*}
  Let \(Z = \{ i \in \{1, \dotsc, 2n\} \mid N_{i} = 0\}\) be the (random) set of
  indices of zeros in \((N_{1}, \dotsc, N_{2n})\). By assumption of the family
  being rich one has for each \(i=1,\dotsc,2n\) that \(r_{i}(0) - r_{i}(1) \geq
  \alpha/2\) and \(p_{i} > \beta/2n\). In particular, for \(i \in Z\) one has
  \(R_{i} = 1/2\) and \(R_{i} - R'_{i} \geq \alpha/2\). Using these
  inequalities, moving the sum inside the expectation, and truncating the sum to
  indices in \(Z\) gives the following lower bound for \(l = 0, \dotsc, 2n-1\)
  \begin{equation} \label{eq:mainlowerstep1}
    \frac{\alpha\beta}{4n}
    \Expect \Big( 
    \sum_{i \in Z}
    \Prob(\PBin(R_{1}, \dotsc, R_{i-1}, R_{i+1}, \dotsc, R_{2n}) = l)
    \Big). 
  \end{equation}
  Note that for \(l \geq 1\),
  \begin{equation*}
    \begin{split}
      &\Prob(\PBin(R_{1}, \dotsc, R_{i-1}, 1/2, R_{i+1},
      \dotsc, R_{2n}) = l) =
      \\&\quad=\Prob(\PBin(R_{1}, \dotsc, R_{i-1},
      R_{i+1}, \dotsc, R_{2n}) = l)/2
      \\&\qquad+ \Prob(\PBin(R_{1}, \dotsc, R_{i-1}, R_{i+1},
      \dotsc, R_{2n}) = l-1)/2.
    \end{split}
  \end{equation*}
  Using this and taking the average of the lower bound in
  Equation~\ref{eq:mainlowerstep1} with \(l\) as well as with \(l\) replaced by \(l-1\)
  gives
  \begin{equation*}
    \begin{split}
      &\delta(\mathcal{E}^{\otimes n}, \mathcal{E}^{\otimes n+1}) \geq
      \\&\quad\geq
      \frac{\alpha\beta}{4n}
      \Expect \Big( 
      \sum_{i \in Z}
      \Prob(\PBin(R_{1}, \dotsc, R_{i-1}, 1/2, R_{i+1}, \dotsc, R_{2n}) = l)
      \Big)
      \\&\quad=
      \frac{\alpha\beta}{4n}
      \Expect \Big( 
      \sum_{i \in Z}
      \Prob(\PBin(R_{1}, \dotsc, R_{i-1}, R_{i}, R_{i+1}, \dotsc, R_{2n}) = l)
      \Big)
      \\&\quad=
      \frac{\alpha\beta}{4n}
      \Expect ( 
      |Z|
      \Prob(\PBin(R_{1}, \dotsc, R_{2n}) = l)
      )
      \\&\quad\geq
      \frac{\alpha\beta}{4}
      \Expect ( 
      \Prob(\PBin(R_{1}, \dotsc, R_{2n}) = l)
      )
    \end{split}
  \end{equation*}
  where the final inequality is due to the fact that \(N_{1} + \dotsb + N_{2n} = n\) implies
  that at least \(n\) out of \(N_{1}, \dotsc, N_{2n}\) must be \(0\). The
  statement now follows by Lemma~\ref{claim:mixedpbin} since the risks
  \(r_{1}(n_{1}), \dotsc, r_{m}(n_{m})\) are non-increasing in the number of
  observations \(n_{1}, \dotsc, n_{m}\). \qedhere{}
\end{proof}

\section{Remaining proofs}
\label{sec:proofs}

The following lemma formalises the notion that deficiency is invariant under
(bijective) reparameterisation or relabelling of the underlying measurable
space. The implicit notion of isomorphism in the lemma is similar to the one used by McCullagh\cite{mccullagh2002}, but does not require, in their language, a common response scale as well as leaving implicit the category of statistical units and designs.
\begin{lemma}\label{lemma:csquare}

  Let \(\mathcal{E} = (X, P_{\theta} ; \theta \in \Theta)\), \(\mathcal{E}' =
  (X', P'_{\psi} ; \psi \in \Psi)\), \(\mathcal{F} = (Y, Q_{\theta} ; \theta \in
  \Theta)\), and \(\mathcal{F}' = (Y', Q'_{\psi} ; \psi \in \Psi)\) be four
  experiments such that there exist a bijection \(\alpha \colon \Psi \to
  \Theta\) and bimeasurable bijections \(\beta_{X} \colon X \to X'\) and
  \(\beta_{Y} \colon Y \to Y'\) making the following squares commute
  \begin{center}
    \begin{tikzcd}
      \Theta \arrow[r, "\alpha"] \arrow[d, "P"] & \Psi \arrow[d, "P'"]\\
      \mathcal{P}(X) \arrow[r, "\beta_{X}^*"] & \mathcal{P}(X')
    \end{tikzcd} %
    \hspace{1em} %
    \begin{tikzcd}
      \Theta \arrow[r, "\alpha"] \arrow[d, "Q"] & \Psi \arrow[d, "Q'"]\\
      \mathcal{P}(Y) \arrow[r, "\beta_{Y}^*"] & \mathcal{P}(Y')\mathrlap{,}
    \end{tikzcd}
  \end{center}

  where \(P\), \(P'\), \(Q\), and \(Q'\) are used to denote that maps \(\theta \mapsto P_{\theta}\), \(\theta \mapsto P'_{\theta}\), \(\theta \mapsto Q_{\theta}\), and \(\theta \mapsto Q'_{\theta}\), respectively.
  
  It then follows that \(\delta(\mathcal{E}, \mathcal{F}) = \delta(\mathcal{E}',
  \mathcal{F}')\).
\end{lemma}
\begin{proof}
  Using the characterisation of deficiency in terms of
  transitions\cite[Theorem~6.4.5]{Torgersen1991} one finds
  \begin{equation}\label{eq:commutativityproof}
    \begin{split} 
    \delta(\mathcal{E}', \mathcal{F}')
    &=
      \min_{T}\sup_{\psi} \| TP'_{\psi} - Q'_{\psi}\|
    \\&=
    \min_{T}\sup_{\theta} \| TP'_{\alpha(\theta)} - Q'_{\alpha(\theta)}\|
    \\&=
    \min_{T}\sup_{\theta} \| T(\beta_{X}^{\ast}P_{\theta}) - \beta_{Y}^{\ast}Q_{\theta}\|
    \\&=
    \min_{T}\sup_{\theta} \| (\beta_{Y}^{\ast} \circ (\beta_{Y}^{\ast})^{-1} \circ T \circ \beta_{X}^{\ast})(P_{\theta}) - \beta_{Y}^{\ast}Q_{\theta}\|
    \\&=
    \min_{T}\sup_{\theta} \| ((\beta_{Y}^{\ast})^{-1} \circ T \circ \beta_{X}^{\ast})(P_{\theta}) - Q_{\theta}\|
    \\&=
    \min_{S}\sup_{\theta} \| SP_{\theta} - Q_{\theta}\|
    \\&=
    \delta(\mathcal{E}, \mathcal{F}),
  \end{split}
  \end{equation}
  where \(S\) and \(T\) range over the collections of transitions between the
  \(L\)-spaces of \(\mathcal{E}\) and \(\mathcal{F}\) (see
  Torgersen[Sections~4.5~and~5.6]\cite{Torgersen1991}). Readers not familiar
  with the general machinery of \(L\)-spaces can rest assured as, for our
  purposes, we need the statement only for experiments \(\mathcal{E}\) and
  \(\mathcal{F}\) sufficiently regular for the above computation to hold with
  \(T\) and \(S\) taken as Markov kernels (see the remark after Theorem~6.4.1 in
  Torgersen\cite{Torgersen1991}).
  
  The second equality in Equation~\ref{eq:commutativityproof} follows by
  \(\alpha\) being a bijection, the third is commutativity of the diagrams, the
  fourth equality follows from the fact that for any \(\mu\) and \(\nu\) one has
  \(\|\beta_{Y}^{\ast}\mu - \beta_{Y}^{\ast}\nu\| = \|\mu - \nu\|\) because
  \(\beta_{Y}\) is a bimeasurable bijection, and the fifth equality follows
  because we claim that \(T \mapsto (\beta_{Y}^{\ast})^{-1} \circ T \circ
  \beta_{X}^{\ast}\) define bijections (of appropriate sets of \(L\)-space
  transitions or Markov-kernels) with inverses \(S \mapsto \beta_{Y}^{\ast}
  \circ S \circ (\beta_{X}^{\ast})^{-1}\). That these maps are inverse follows
  directly from the functoriality of the assignment \(\beta \mapsto
  \beta^{\ast}\). It remains to see that they map \(L\)-space transitions to
  \(L\)-space transitions and/or maps induced by Markov kernels to maps induced
  by Markov kernels.

  For the former we will prove that \(\beta_{X}^{\ast}\) maps the \(L\)-space of
  \(\mathcal{E}\) into the \(L\)-space of \(\mathcal{E}'\). The same argument
  will prove the analogous statement for \((\beta_{Y}^{\ast})^{-1}\). Together
  they imply the statement. Since the \(L\)-spaces are simply spaces of measures, 
  this is well defined even though \(\mathcal{E}\) and \(\mathcal{E}'\) have
  different parameter spaces. For any family \((a_{\theta})_{\theta \in
    \Theta}\) such that \(a_{\theta} \neq 0\) only on a countable sequence
  \(\theta_{1}, \theta_{2}, \dotsc\) we need to show that if some \(\mu \ll
  \sum_{\theta} a_{\theta}P_{\theta}\) then \(f^{\ast}_{X}(\mu) \ll
  \sum_{\psi}b_{\psi}P'_{\psi}\) for some \((b_{\psi})_{\psi \in \Psi}\) that is
  non-zero on at most a countable set. Since \(\beta_{X}^{\ast}\) preserves
  probability measures, it is bounded and therefore continuous. Letting
  \(b_{\psi} = a_{\alpha^{-1}(\psi)}\) we have
  \begin{equation*}
    \beta_{X}^{\ast}(\sum_{\theta}a_{\theta}P_{\theta})
    =
    \sum_{\theta}a_{\theta}\beta_{X}^{\ast}(P_{\theta})
    =
    \sum_{\theta}a_{\theta}P'_{\alpha(\theta)}
    =
    \sum_{\theta}a_{\theta}P'_{\alpha(\theta)}
    =
    \sum_{\psi}b_{\psi}P'_{\psi}.
  \end{equation*}
  But it is immediate from the definition of \(\beta_{X}^{\ast}\) that
  \(\beta_{X}^{\ast}(\mu) \ll
  \beta_{X}^{\ast}(\sum_{\theta}a_{\theta}P_{\theta})\), so the result follows.

  In case one wishes to restricts to the case of Markov kernels we note that if
  \(T(P)(A) = \int K_{x}(A) \, P(dx)\) for some Markov kernel \(K\) then
  \begin{equation*}
    \begin{split}
      ((\beta_{Y}^{\ast})^{-1} \circ T \circ \beta_{X}^{\ast})(P)(A)
      &=
      (T(\beta_{X}^{\ast}(P)))(\beta_{Y}(A))
      \\&=
      \int K_{x}(\beta_{Y}(A))\,\beta_{X}^{\ast}(P)(dx)
      \\&=
      \int K_{\beta_{X}^{-1}(x)}(\beta_{Y}(A))\,P(dx).
    \end{split}
  \end{equation*}
  Finally \((x, A) \mapsto K_{\beta_{X}^{-1}(x)}(\beta_{Y}(A))\) is measurable
  in \(x\) for each fixed \(A\) because \(\beta_{X}^{-1}\) is (bi)measurable and
  a Probability measure for each fixed \(x\) because \(\beta_{Y}\) is a
  (bi)measurable bijection. \qedhere{}
\end{proof}

\begin{proof}[Proof~of~Lemma~\ref{lemma:hypercube}]

  By Lemma~\ref{lemma:csquare} it suffices to produce for each \(n =
  1,2,\dotsc\) a bimeasurable bijection \(\beta = \beta_{n} \colon X^{\otimes n}
  \to (X_{1} \oplus \dotsb \oplus X_{k})^{\otimes n}\) such that
  \(\beta_{n}^{\ast}(P_{\theta}^{\otimes n}) = P_{\vartheta(\theta)}^{\otimes
    n}\) where \(P_{\vartheta(\theta)} =
  \sum_{i=1}^{k}P_{\theta}(X_{i})P_{i,\theta}\). It is sufficient to do this for
  \(n = 1\), the general case follows by applying the transformation to each
  component.

  Define \(\beta \colon X \to X_{1} \oplus \dotsb \oplus X_{k}\) by, for each
  \(i = 1, \dotsc, k\), \(\beta(x) = \iota_{i}(x)\) when \(x \in X_{i}\) where
  \(\iota_{i}\) is the natural injection \(X_{i} \to X_{1} \oplus \dotsb \oplus
  X_{k}\). Measurability of \(\beta\) follows from \(\iota_{1}, \dotsc
  \iota_{k}\) being measurable and \(X_{1}, \dotsc, X_{k}\) forming a measurable
  partition of \(X\). Since \(\iota_{1}, \dotsc, \iota_{k}\) are injective and
  \(\iota_{1}(X_{1}), \dotsc, \iota_{k}(X_{k})\) are disjoint it follows that
  \(\beta\) is injective. Since \(\bigcup_{i=1}^{k}\iota_{n}(X_{n}) =
  \bigoplus_{i=1}^{k}X_{n}\) it follows that \(\beta\) is surjective, and hence
  bijective.

  For the inverse \(\beta^{-1}\) to be measurable we need that \(\beta(A)\) be
  measurable for each measurable \(A\). By definition \(A = \bigcup_{i=1}^{k}(A
  \cap \iota_{n}(X_{n}))\) is measurable if and only if \(A \cap
  \iota_{1}(X_{1}) = \iota_{1}(A_{1}), \dotsc, A \cap \iota_{k}(X_{k}) =
  \iota_{k}(A_{k})\) for some measurable \(A_{1}, \dotsc, A_{k}\). But then
  \(\beta(A) = \bigcup_{i=1}^{k}\beta(\iota_{n}(A_{n})) =
  \bigcup_{i=1}^{k}A_{n}\) which is measurable since \(A_{1}, \dotsc, A_{k}\)
  are all measurable.
  
  It remains to prove that \(\beta^{\ast}(P_{\theta}) = P_{\vartheta(\theta)}\).
  To be equal it is sufficient that they agree on \(\iota_{i}(A)\) for
  \(i=1,\dotsc,k\) and \(A\) being \(X_{n}\)-measurable, since these sets
  generate the measurable sets of \(X_{1} \oplus \dotsc \oplus X_{k}\). But
  \begin{equation*}
    \begin{split}
      \beta^{\ast}(P_{\theta})(\iota_{i}(A))
      &=
      P_{\theta}(\beta^{-1}(\iota_{i}(A)))
      =
      P_{\theta}(A)
      \\&=
      P(X_{i})P_{i,\theta}(A)
      =
      P_{\vartheta(\theta)}(\iota_{n}(A)). \qedhere
    \end{split}
  \end{equation*}
\end{proof}

\begin{proof}[Proof of Lemma~\ref{lemma:mixedpower}]
  Define \(I \colon \bigoplus_{i=1}^{m}X_{i} \to \mathcal{D}(\{ 1, \dotsc, m
  \})\) by \(I(\iota_{i}(x)) = i\) for \(i = 1, \dotsc, m\), any \(x \in
  X_{i}\), and where \(\iota_{i} \colon X_{i} \to \bigoplus_{i=1}^{m}X_{i}\) is
  the natural injection. Since, by definition, \(\iota_{1}(X_{1}), \dotsc,
  \iota_{m}(X_{m})\) is a measurable partition of \(\bigoplus_{i=1}^{m}X_{i}\)
  we have that \(I\) is measurable. For any natural number \(k\) let \(S_{k}\)
  denote the group of permutations on \(\{1, \dotsc, k\}\). Given a vector of
  integers \(\kappa = (\kappa_{1}, \dotsc, \kappa_{k})\) there exists a unique
  stable sorting permutation \(\sigma_{\kappa} \in S_{k}\) defined by
  \(\kappa_{\sigma_{\kappa}(1)} \leq \kappa_{\sigma_{\kappa}(2)} \leq \dotsb
  \leq \kappa_{\sigma_{\kappa}(k)}\) and if \(i \leq j\) and \(\kappa_{i} =
  \kappa_{j}\) then \(\sigma_{\kappa}(i) \leq \sigma_{\kappa}(j)\). Note that
  all maps \(\mathcal{D}(\{1, \dotsc, k\})^{\otimes n} \to \mathcal{D}(S_{m})\)
  are measurable. Combining the above gives that the map \((x_{1}, \dotsc,
  x_{m}) \mapsto \sigma_{(I(x_{1}), \dotsc, I(x_{m}))}\) is measurable. Moreover,
  define \(c \colon \mathcal{D}(\Z)^{\otimes n} \to \mathcal{D}(\Z)^{\otimes
    m}\) by \(c(l) = (c_{1}(l), \dotsc, c_{m}(l))\) where \(c_{i}(l_{1}, \dotsc,
  l_{n}) = |\{ j \mid l_{j} = i\}|\). Combining the above we may define the
  measurable map \(\eta \colon (\bigoplus_{i=1}^{m}X_{i})^{\otimes n} \to
  \bigoplus_{n_{1} + \dotsb + n_{m} = n} X_{1}^{\otimes n_{1}} \otimes \dotsb
  \otimes X_{m}^{\otimes n_{m}}\) by
  \begin{equation*}
    \begin{split}
      &\eta(\iota_{i_{1}}(x_{1}), \dotsc, \iota_{i_{n}}(x_{n})) =
      \\&\quad= \iota'_{c(I(x_{1}), \dotsc, I(x_{n}))}(x_{\sigma_{I(x_{1}),
          \dotsc, I(x_{n})}(1)}, \dotsc, x_{\sigma_{I(x_{1}),
          \dotsc, I(x_{n})}(n)}
      )
    \end{split}
  \end{equation*}
  where for any partition \(n_{1}, \dotsc, n_{m}\) of \(n\) we let
  \(\iota'_{n_{1}, \dotsc, n_{m}}\) denote the injection \(X_{1}^{\otimes n_{1}}
  \otimes \dotsb \otimes X_{m}^{\otimes n_{m}} \to \bigoplus_{n_{1} + \dotsb +
    n_{m} = n} X_{1}^{\otimes n_{1}} \otimes \dotsb \otimes X_{m}^{\otimes
    n_{m}}\). The map \(\eta\) simply sorts a vector in
  \((\bigoplus_{i=1}^{m}X_{i})^{\otimes n}\) according to which subspace
  \(\iota_{1}(X_{1}), \dotsc, \iota_{m}(X_{n})\) each component lies in.
  
  Fix some \(\theta \in \Theta\), let \(A = (A_{1}, \dotsc, A_{n})\) be a vector
  of jointly independent random variables such that \(\Prob(A_{i} = j) =
  p_{j}\), \(\xi_{1,1}, \dotsc, \xi_{1,n}\) be jointly independent
  \(P_{1,\theta}\)-distributed, \ldots, and \(\xi_{m,1}, \dotsc, \xi_{m,n}\) be
  jointly independent \(P_{m, \theta}\)-distributed, all jointly independent of
  each other. By construction \(\xi = (\xi_{1}, \dotsc, \xi_{n}) =
  (\iota_{A_{1}}(\xi_{A_{1},1}), \dotsc, \iota_{A_{n}}(\xi_{A_{n}, n}))\) is
  \((\sum_{i=1}^{m}p_{i}P_{i,\theta})^{\otimes n}\)-distributed so that
  \(\eta(\xi)\) has law \(\eta^{\ast}((\sum_{i=1}^{m}p_{i}P_{i,\theta})^{\otimes
    n})\). Since \(I(\xi_{i}) = A_{i}\) it immediately follows that
  \begin{equation*}
    (c_{1}(I(\xi_{1}), \dotsc, I(\xi_{n})), \dotsc, c_{m}(I(\xi_{1}), \dotsc,
    I(\xi_{n}))) = (c_{1}(A), \dotsc, c_{m}(A))
  \end{equation*}
  is distributed according to \(\Mult(n, p_{1}, \dotsc, p_{m})\) so that \(\sigma_{I(\xi_{1}),\dotsc,I(\xi_{n})} =
  \sigma_{A}\)

  For some fixed vector \(a = (a_{1}, \dotsc, a_{n}) \in \{1,\dotsc,m\}^{n}\)
  consider the conditional distribution \(\xi\) on the event \(A = a\). Since
  \(A_{1}, \dotsc, A_{n}\) are mutually independent and independent of
  \(\xi_{1,1}, \xi_{1,2}, \dotsc, \xi_{m,n-1}, \xi_{m,n}\) we have that
  \(\iota_{A_{1}}(\xi_{A_{1},1}), \dotsc, \iota_{A_{n}}(\xi_{A_{n}, n})\) are
  still mutually independent conditional on \(A=a\) with the (conditional)
  marginal law of \(\xi_{A_{i},i}\) being \(P_{a_{i},\theta}\). Since
  \(\eta(\xi) = \iota'_{c(A)}(\xi_{\sigma_{A}(1)},\dotsc,\xi_{\sigma_{A}(n)})\)
  we have that the conditional law of \(\eta(\xi)\) on \(A = a\) is
  \((\iota'_{c(a)})^{\ast}(P_{a_{\sigma_{a}}(1),\theta} \otimes \dotsb \otimes
  P_{a_{\sigma_{a}}(n),\theta}) = (\iota'_{c(a)})^{\ast}(P_{1,\theta}^{\otimes
    c_{1}(a)} \otimes \dotsb \otimes P_{m,\theta}^{\otimes c_{m}(a)})\). By the
  law of total probability
  \begin{equation*}
    \eta^{\ast}((\sum_{i=1}^{m}p_{i}P_{i,\theta})^{\otimes n}) =
    \sum_{n_{1}, \dotsc, n_{m}} q(n_{1}, \dotsc, n_{m}) \, (\iota'_{n_{1},\dotsc,n_{m}})^{\ast}(P_{1,\theta}^{\otimes n_{1}}
    \otimes \dotsb \otimes P_{m,\theta}^{\otimes n_{m}})
  \end{equation*}
  recalling that \(q(n_{1}, \dotsc, n_{m}) = \Prob(\Mult(n, p_{1}, \dotsc,
  p_{m}) = (n_{1},\dotsc,n_{k}))\).

  It follows that the left hand \((\sum_{i=1}^{m}p_{i}\mathcal{E}_{i} )^{\otimes
    n}\) is at least as informative as the right-hand \(\sum_{n_{1}, \dotsc,
    n_{m}} q(n_{1}, \dotsc, n_{m}) \, \mathcal{E}_{1}^{\otimes n_{1}} \otimes
  \dotsb \otimes \mathcal{E}_{m}^{\otimes n_{m}}\).

  For the converse, let
  \begin{equation*}
    \eta' \colon \bigoplus_{n_{1} + \dotsb + n_{m} = n}
    X_{1}^{\otimes n_{1}} \otimes \dotsb \otimes X_{m}^{\otimes n_{m}} \to
    \Bigl(\bigoplus_{i=1}^{m}X_{i}\Bigr)^{\otimes n}
  \end{equation*}
  be defined by
  \begin{equation*}
    \begin{split}
      &\eta'(\iota_{n_{1},\dotsc,n_{m}}(x_{1}, \dotsc, x_{n})) =
      \\&\quad=
      (\iota_{1}(x_{1}), \dotsc,
      \iota_{1}(x_{n_{1}}), \dotsc, \iota_{m}(x_{n-n_{m}}), \dotsc,
      \iota_{m}(x_{n})).
    \end{split}
  \end{equation*}
  \(\eta'\) is measurable on each part \(\iota_{n_{1},\dotsc,n_{m}}(X^{\otimes
    n_{1}} \otimes \dotsb \otimes X^{\otimes n_{m}})\) since \(\iota_{1},
  \dotsc, \iota_{k}\) are measurable, and hence measurable on all of
  \(\bigoplus_{n_{1} + \dotsb + n_{m} = n} X_{1}^{\otimes n_{1}} \otimes \dotsb
  \otimes X_{m}^{\otimes n_{m}}\!\).

  It remains now only to note that \(\eta'(\eta(\xi))\) is equal to
  \begin{equation*}
    (\iota_{1}(\xi_{\sigma_{A}(1)}), \dotsc, \iota_{1}(\xi_{\sigma_{A}(m_{1}(A))}), \dotsc, \iota_{m}(\xi_{\sigma_{A}(n-m_{k}(A))}), \dotsc, \iota_{k}(\xi_{\sigma_{A}(n)}))
  \end{equation*}
  which is a permutation of \(\xi\). In particular, \(\eta'(\eta(\xi))\) and
  \(\xi\) have the same empirical measure. Since the empirical measure is a
  sufficient statistic for \((\sum_{i=1}^{m}p_{i}\mathcal{E}_{i} )^{\otimes n}\)
  the result follows. \qedhere{}
\end{proof}

In order to prove Lemma~\ref{lemma:multitestbound} we will need some basic
results concerning the risks of a certain type of multiple decision problems.

\begin{lemma}\label{lemma:multiprob}

  Let \(\mathcal{E}_{1} = (X_{1}, P_{1,\theta} ; \theta \in \Theta_{1}), \dotsc,
  \mathcal{E}_{m} = (X_{m}, P_{m,\theta} ; \theta \in \Theta_{m})\) be a
  sequence of experiments on finite parameter spaces \(\Theta_{1}, \dotsc,
  \Theta_{m}\). Also let \((A_{1}, L_{1}), \dotsc, (A_{m}, L_{m})\) be a
  sequence of finite 0-1-decision problems on \(\mathcal{E}_{1}, \dotsc,
  \mathcal{E}_{m}\), respectively.

  Fix some \(l \in \N\), define the experiment \(\mathcal{E} = (X_{1} \otimes
  \dotsb \otimes X_{m}, P_{1,\theta_{1}} \otimes \dotsb \otimes P_{m,\theta_{m}}
  ; (\theta_{1}, \dotsc, \theta_{m}) \in \Theta_{1} \times \dotsb \times
  \Theta_{m})\), and define a 0-1-decision problem \((A, L)\) by \(A = A_{1}
  \times \dotsb \times A_{m}\) and
  \begin{equation*}
    L(a, \theta) =
    \begin{cases}
      0 & L_{1}(a_{1}, \theta_{1}) + \dotsb + L_{m}(a_{m},
      \theta_{m}) < l
      \\
      1 & L_{1}(a_{1}, \theta_{1}) +
      \dotsb + L_{m}(a_{m}, \theta_{m}) \geq l
    \end{cases}.
  \end{equation*}

  Given decision procedures \(\varphi_{1}, \dotsc, \varphi_{m}\) for \((A_{1},
  L_{1}), \dotsc, (A_{m}, L_{m})\) on observing \(\mathcal{E}_{1}, \dotsc,
  \mathcal{E}_{m}\), define the decision procedure \(\varphi = \varphi_{1}
  \otimes \dotsb \otimes \varphi_{m}\) for \((A, L)\) on observing
  \(\mathcal{E}\).
  
  If \(\varphi_{1}, \dotsc, \varphi_{m}\) have risks \(r_{1}, \dotsc, r_{m}\) at
  some \(\theta_{1} \in \Theta, \dotsc, \theta_{m} \in \Theta_{m}\) on observing
  \(\mathcal{E}_{1}, \dotsc, \mathcal{E}_{m}\) then \(\varphi\) has risk
  \(\Prob(\PBin(r_{1}, \dotsc, r_{m}) \geq l)\) at \((\theta_{1}, \dotsc,
  \theta_{m}) \in \Theta_{1} \times \dotsb \times \Theta_{m}\) on
  observing \(\mathcal{E}\).

  If \(\varphi_{1}, \dotsc, \varphi_{m}\) are Bayes with respect to some priors
  \(\pi_{1}, \dotsc, \pi_{m}\) on \(\Theta_{1}, \dotsc, \Theta_{m}\) on
  observing \(\mathcal{E}_{1}, \dotsc, \mathcal{E}_{m}\) then \(\varphi\) is
  Bayes with respect to the prior \(\pi = \pi_{1} \otimes \dotsb \otimes
  \pi_{m}\) on \(\Theta_{1} \times \dotsb \times \Theta_{m}\) on observing
  \(\mathcal{E}\).
\end{lemma}

\begin{proof}
  Assume \(\varphi_{1}, \dotsc, \varphi_{m}\) have risk \(r_{1}, \dotsc, r_{m}\)
  at some \(\theta_{1} \in \Theta_{1}, \dotsc, \theta_{m} \in \Theta_{m}\). For
  \(i = 1, \dotsc, m\) denote \(B_{i} = \{ a \in A_{i} \mid L_{i}(a, \theta_{i})
  = 1\}\). Let \(\xi_{1}, \dotsc, \xi_{m}\) be independent random variables
  distributed according to \(P_{1,\theta_{1}}, \dotsc, P_{m,\theta_{m}}\),
  respectively. By definition \(r_{i} = \Expect(L(\varphi_{i}(\xi_{i}),
  \theta_{i})) = \Prob(\varphi_{i}(\xi_{i}) \in B_{i})\) for \(i=1, \dotsc, m\).

  Since \(L(\varphi_{1}(\xi_{1}), \theta_{1}), \dotsc, L(\varphi_{m}(\xi_{m}),
  \theta_{m})\) are independent Bernoulli we have
  \begin{equation*}
    \begin{split}
      &\Expect(L(\varphi_{1}(\xi_{1}), \dotsc, \varphi_{m}(\xi_{m});
      \theta_{1}, \dotsc, \theta_{m}))
      =\\
      &\quad=
      \Prob(L(\varphi_{1}(\xi_{1}), \theta_{1}) +
      \dotsb + L_{m}(\varphi_{m}(\xi_{m}), \theta_{m}) \geq l)
      \\&\quad=
      \Prob(\PBin(r_{1}, \dotsc, r_{m}) \geq l).
    \end{split}
  \end{equation*}

  Assume, now, instead that \(\varphi_{1}, \dotsc, \varphi_{m}\) are Bayes with
  respect to priors \(\pi_{1}, \dotsc, \pi_{m}\) on \(\Theta_{1}, \dotsc,
  \Theta_{m}\), respectively. Let \((\theta_{1}, \xi_{1}), \dotsc, (\theta_{m},
  \xi_{m})\) be mutually independent pairs with \(\theta = (\theta_{1}, \dotsc,
  \theta_{m})\) distributed according to \(\pi = \pi_{1} \otimes \dotsb \otimes
  \pi_{m}\) and \((\xi_{1}, \dotsc, \xi_{m})\) distributed according to
  \(P_{\theta} = P_{1, \theta_{1}} \otimes \dotsb \otimes P_{m,\theta_{m}}\)
  conditional on \(\theta_{1}, \dotsc, \theta_{m}\).

  Denote by \(P_{\pi_{i}} = \sum_{t_{i}} \pi_{i}(t_{i}) P_{i,t_{i}}\) the
  marginal distribution of \(\xi_{i}\) and \(P_{\pi} = \sum_{t} \pi(t)P_{t} =
  P_{1,\pi_{1}} \otimes \dotsb \otimes P_{m,\pi_{m}}\) the marginal distribution
  of \((\xi_{1}, \dotsc, \xi_{n})\).

  Since the parameter spaces are finite each family \((P_{i,\theta_{i}})_{\theta
    \in \Theta}\) is dominated and we may define posterior distributions
  \(\pi^{x}\) for \(x = (x_{1}, \dotsc, x_{m})\) such that for any integrable
  \(f\) one has\cite[Proposition~3.32]{Liese2008}
  \begin{equation*}
    \int f(x, t) \,P_{t}(dx) \,\pi(dt) = \int f(x, t)
    \,\pi^{x}(dt) \,P_{\pi}(dx).
  \end{equation*}
  Moreover, by independence it may be chosen such that \(\pi^{x} =
  \pi_{1}^{x_{1}} \dotsb \pi_{m}^{x_{m}}\) where for integrable \(f_{i}\) one
  has
  \begin{equation*}
    \int f(x_{i}, t_{i}) \,P_{i,t_{i}}(dx_{i}) \,\pi_{i}(dt_{i}) = \int
    f_{i}(x_{i}, t_{i}) \,\pi_{i}^{x_{i}}(dt_{i}) \,P_{i,\pi_{i}}(dx_{i}).
  \end{equation*}
  Bayes procedures \(\delta\) for \(\pi\) are given by \(\delta(x)\) being
  minimisers of the posterior risks \(a \mapsto \pi^{x}(t \mapsto L(a, x))\) for
  \(P_{\pi}\)-almost every \(x\)\cite[Proposition~3.37]{Liese2008}.
  
  The above factorisation implies that the distribution of \(L_{1}(a_{1},
  \theta_{1}) + \dotsb + L_{m}(a_{m}, \theta_{m})\) under \(\pi^{x}\) is given
  by \(\PBin(\rho^{x_{1}}_{1}(a_{1}), \dotsc, \rho^{x_{m}}_{m}(a_{m}))\) where
  \(\rho^{x_{i}}_{i}(a_{i}) = \pi_{i}^{x_{i}}(\Theta_{a_{i},i})\) and
  \(\Theta_{a_{i},i} = \{ t_{i} \in \Theta_{i} \mid L_{i}(a_{i}, t_{i}) = 1\}\).
  
  Since \(\PBin(p_{1},\dotsc,p_{m})\) is monotone increasing in (the usual)
  stochastic (dominance) order (see for example \cite[Chapter~1]{Shaked2007})
  with respect to \(p_{1}, \dotsc, p_{m}\) it follows that minimising
  \(\Prob(\PBin(p_{1},\dotsc,p_{m}) \geq l)\) can be done by minimising \(p_{1},
  \dotsc, p_{m}\) separately. But by assumption \(\varphi_{1}, \dotsc,
  \varphi_{m}\) are Bayes for \((A_{1}, L_{1}), \dotsc, (A_{m}, L_{m})\) with
  respect to \(\pi_{1}, \dotsc, \pi_{m}\), respectively. This means exactly that
  they minimise \(\rho^{x_{1}}_{1}(a_{1}), \dotsc, \rho^{x_{m}}_{m}(a_{m})\),
  respectively, for \((P_{\pi_{1}} \otimes \dotsb \otimes P_{\pi_{m}})\)-almost
  all \((x_{1}, \dotsc, x_{m})\). The result follows. \qedhere{}
\end{proof}

We need a similar result for mixture experiments. This is essentially just the
statement that Bayes decisions satisfy the so called conditionality principle,
specialised to our situation.

\begin{lemma}\label{lemma:mixoptimal}

  Let \(\mathcal{E}_{1} = (X_{1}, P_{1,\theta} ; \theta \in \Theta), \dotsc,
  \mathcal{E}_{m} = (X_{m}, P_{m,\theta} ; \theta \in \Theta)\) be a collection
  of experiments on some finite parameter space \(\Theta\). Let \((A, L)\) be
  some finite decision problem with corresponding decision procedures
  \(\varphi_{1}, \dotsc, \varphi_{m}\) on observing \(\mathcal{E}_{1}, \dotsc,
  \mathcal{E}_{m}\), respectively. For any convex coefficients \((p_{1}, \dotsc,
  p_{m}) \in \Delta_{m}\) define \(\varphi = \varphi_{1} \oplus \dotsb \oplus
  \varphi_{m}\) as a decision procedure for \((A, L)\) on observing the mixture
  experiment \(\sum_{i=1}^{m} p_{i} \mathcal{E}_{i}\).

  If \(\varphi_{1}, \dotsc, \varphi_{m}\) have risks \(r_{1}, \dotsc, r_{m}\) at
  \(\theta \in \Theta\) then \(\varphi\) has risk \(\sum_{i} p_{i} r_{i}\) at
  \(\theta\).

  If \(\varphi_{1}, \dotsc, \varphi_{m}\) are Bayes with respect to the same
  prior \(\pi\) then \(\varphi\) is Bayes with respect \(\pi\).
\end{lemma}
\begin{proof}
  That \(\varphi_{1}, \dotsc, \varphi_{m}\) have risk \(r_{1}, \dotsc, r_{m}\)
  for some particular \(\theta \in \Theta\) then it is immediate from the
  definition of \(\sum_{i=1}^{m}p_{i}\mathcal{E}_{i}\) that \(\varphi\) has risk
  \(\sum_{i=1}^{m}p_{i}r_{i}\) for the same \(\theta \in \Theta\).

  The fact that \(\varphi\) is Bayes for \(\pi\) whenever \(\varphi_{1}, \dotsc,
  \varphi_{m}\) are Bayes for \(\pi\) follows similarly to the proof of
  Lemma~\ref{lemma:multiprob}. There exists posterior distributions \(\pi^{x}\)
  with the property that \(\pi^{\iota_{i}(x)} = \pi_{i}^{x}\) where
  \(\pi_{i}^{x}\) is a posterior distribution on \(\Theta\) on observing \(x\)
  from \(\mathcal{E}_{i}\). Minimising the posterior risk is therefore
  equivalent to minimising the posterior with respect to \(\pi^{x}\) can
  therefore be done by minimising each of \(\pi_{1}^{x}, \dotsc, \pi_{m}^{x}\)
  separately. But by assumption, as in Lemma~\ref{lemma:multiprob}, this is what
  \(\varphi_{1}, \dotsc, \varphi_{m}\) do. \qedhere{}
\end{proof}

\begin{proof}[Proof~of~Lemma~\ref{lemma:multitestbound}]

  Let experiments \(\mathcal{E}_{1}, \dotsc, \mathcal{E}_{m}, \mathcal{E},
  \mathcal{E}'\), with corresponding parameter pairs \(\theta_{1,0},
  \theta_{1,1}\), \ldots, \(\theta_{m,0}, \theta_{m,1}\), priors \(\pi_{1},
  \dotsc, \pi_{m}\), Bayes risks \(r_{1}(n), \dotsc, r_{m}(n)\), and integer \(l
  \in \{1, \dotsc, m\}\) be as in the statement. For any \(n \in \N\) and \(i =
  1, \dotsc, m\) let \(\varphi_{i,n}\) be Bayes procedures for testing
  \(P_{i,0}^{\otimes n}\) against \(P_{i,1}^{\otimes n}\) under the prior
  \(\pi_{i}\), with the corresponding Bayes risks \(r_{i}(n)\). Let also
  \begin{equation*}
    \begin{split}
      \makebox[\widthof{\(\mathcal{E}'_{m}\)}][c]{\(\mathcal{E}'_{1}\)} &= (\makebox[\widthof{\(X_{m}\)}][c]{\(X_{1}\)},
      \makebox[\widthof{\(P_{m,\theta}\)}][c]{\(P_{1,\theta}\)} ; \makebox[\widthof{\(\theta \in
        \{\theta_{m,0}, \theta_{m,1}\}\)}][c]{\(\theta \in
        \{\theta_{1,0}, \theta_{1,1}\}\)})
      \\
      &\vdotswithin{=}
      \\
      \mathcal{E}'_{m} &= (X_{m}, P_{m,\theta} ; \theta \in \{\theta_{m,0},
      \theta_{m,1}\}).
    \end{split}
  \end{equation*}
  For any \(n_{1}, \dotsc, n_{m}\) apply Lemma~\ref{lemma:multiprob} to
  experiments \((\mathcal{E}'_{1})^{\otimes n_{1}}\), \ldots{},
  \((\mathcal{E}'_{m})^{\otimes n_{m}}\), procedures \(\varphi_{1,n_{1}}\),
  \ldots{}, \(\varphi_{m,n_{m}}\), and the testing problems \((\{0, 1\},
  L_{1})\), \ldots{}, \((\{0, 1\}, L_{m})\) where \(L_{i}(j, \theta_{i,k})\) is
  \(0\) if \(j = k\) and \(1\) otherwise. This gives that \(\varphi_{1,n_{1}}
  \otimes \dotsb \otimes \varphi_{m,n_{m}}\) is a Bayes procedure with respect
  to the prior \(\pi = \pi_{1} \otimes \dotsb \otimes \pi_{m}\) on observing the
  experiment \((X_{1}^{\otimes n_{1}} \otimes \dotsb \otimes X_{m}^{\otimes
    n_{m}}, P_{1,\theta_{1}}^{\otimes n_{1}} \otimes \dotsb \otimes
  P_{1,\theta_{m}}^{\otimes n_{m}}; (\theta_{1}, \dotsc, \theta_{m}) \in
  \{\theta_{1,0}, \theta_{1,1}\} \times \dotsb \times \{\theta_{m,0},
  \theta_{m,1}\})\) and that it has Bayes risk \(\Prob(\PBin(r_{1}(n_{1}),
  \dotsc, r_{m}(n_{m})) \geq l)\). This experiment is exactly the restriction of
  \(\mathcal{E}_{1}^{\otimes n_{1}} \otimes \dotsb \otimes
  \mathcal{E}_{m}^{\otimes n_{m}}\) to \(\{\theta_{1,0}, \theta_{1,1}\} \times
  \dotsb \times \{\theta_{m,0}, \theta_{m,1}\}\).

  Applying Lemma~\ref{lemma:mixoptimal} it follows that for any probability mass
  function \(p\) on \(\N^{m}\) one has that \(\varphi =
  \bigoplus_{n_{1},\dotsc,n_{m}}\varphi_{n_{1},\dotsc,n_{m}}\) is Bayes with
  respect to the prior \(\pi\) on observing \(\sum_{n_{1}+\dotsc+n_{m}} p(n_{1},
  \dotsc, n_{m})(\mathcal{E}_{1}^{\otimes n_{1}} \otimes \dotsb \otimes
  \mathcal{E}_{m}^{\otimes n_{m}})\) and that the corresponding Bayes risks are
  \begin{equation} \label{eq:finalbayesriskformula}
    \begin{split}
      \sum_{n_{1}, \dotsc, n_{m}}p(n_{1}, \dotsc, n_{m})&\Prob(\PBin(r_{1}(n_{1}), \dotsc, r_{m}(n_{m})) \geq l)
      =\\
      = \Expect(&\Prob(\PBin(r_{1}(N_{1}), \dotsc, r_{m}(N_{m})) \geq l))
    \end{split}
  \end{equation}
  where \((N_{1}, \dotsc, N_{m})\) are distributed according to \(p\). In
  particular this holds for \(p = q\) and \(p = q'\).

  It remains only to recall that the deficiency is bounded from below by any
  difference in achievable Bayes risk, for finitely supported priors and finite
  normalised decision problems. Applying this to
  Equation~\ref{eq:finalbayesriskformula} with \(p = q\) and \(p = q'\) yields
  the result. \qedhere{}
\end{proof}

\begin{proof}[Proof of Lemma~\ref{lemma:pbindiff}]
  Since \(\PBin(p_{1}, \dotsc, p_{n})\) is invariant under permutation of the
  parameters \((p_{1}, \dotsc, p_{n})\) we may assume without loss of generality
  that \(i = n+1\).
  
  Let \(U_{1}, \dotsc, U_{n+1}\) be independent and uniform on \([0, 1]\) and
  \(N = \mathbf{1}(U_{1} < p_{1}) + \dotsb + \mathbf{1}(U_{n} < p_{n}) \sim \PBin(p_{1}, \dotsc, p_{n})\), \(M =
  N + \mathbf{1}(U_{n+1} < p_{n+1}) \sim \PBin(p_{1}, \dotsc, p_{n}, p_{n+1})\), and \(M' = N + \mathbf{1}(U_{n+1} <
  p'_{n+1}) \sim \PBin(p_{1}, \dotsc, p_{n}, p'_{n+1})\). By the law of total probability we have
  \begin{align*}
      \Prob(M \geq l)
      &=
      \Prob(N \geq l-1 \mid U_{n+1} < p'_{n+1})p'_{n+1}
      \\&\quad+
      \Prob(N \geq l-1 \mid p'_{n+1} < U_{n+1} < p_{n+1})(p_{n+1}-p'_{n+1})
      \\&\quad+
      \Prob(N \geq l \mid p_{n+1} < U_{n+1})(1-p_{n+1})
      \\&=
      \Prob(N \geq l-1)p'_{n+1}
      \\&\quad+
      \Prob(N \geq l-1)(p_{n+1}-p'_{n+1})
      \\&\quad+
      \Prob(N \geq l)(1-p_{n+1})
      \intertext{and}
      \Prob(M' \geq l)
      &=
      \Prob(N \geq l-1)p'_{n+1}
      \\&\quad+
      \Prob(N \geq l)(p_{n+1}-p'_{n+1})
      \\&\quad+
      \Prob(N \geq l)(1-p_{n+1}).
  \end{align*}
  Combining the above
  \begin{equation*}
    \begin{split}
      &\Prob(\PBin(p_{1}, \dotsc, p_{n+1}) \geq l)
      -
      \Prob(\PBin(p_{1}, \dotsc, p'_{n+1}) \geq l)
      =
      \\&\quad=
      \Prob(M \geq l)
      -
      \Prob(M' \geq l)
      \\&\quad=
      (\Prob(N \geq l-1)
      -
      \Prob(N \geq l))(p_{n+1}-p'_{n+1})
      \\&\quad=
      \Prob(N = l-1)(p_{n+1}-p'_{n+1}).
    \end{split}
  \end{equation*}
  Since \(N \sim \PBin(p_{1}, \dotsc, p_{n})\) this concludes the proof.
\end{proof}

\begin{proof}[Proof~of~Lemma~\ref{claim:mixedpbin}]
  This follows from a double concentration argument, first showing that the
  (random) mean is concentrated and then using that conditionally on the
  (random) parameters the Poisson-binomial quantity is concentrated around its
  mean.

  For any \emph{fixed} vector \(q_{1}, \dotsc, q_{m}\) we have by Hoeffding's
  inequality for any \(t > 0\) and denoting \(b_{1} = \sqrt{t/m}\) that
  \begin{equation*}
    \Prob(\mu - b_{1} < \PBin(q_{1}, \dotsc, q_{m}) < \mu
    + b_{1}) \geq 1 - 2e^{-2t}
  \end{equation*}
  where \(\mu = \Expect(\PBin(q_{1}, \dotsc, q_{m})) = q_{1} + \dotsb + q_{m}\).

  Consider a random vector \(N = (N_{1}, \dotsc, N_{m}) \sim \Mult(n, p_{1},
  \dotsc, p_{m})\). Any multinomial vector is negatively
  associated\cite{Dubhashi1998}. Since the functions \(f_{1}, \dotsc, f_{m}\)
  are either all decreasing or all increasing we have that the vector
  \((f_{1}(N_{1}), \dotsc, f_{m}(N_{m}))\) is also negatively
  associated\cite[Proposition~8]{Dubhashi1998}. This in turn implies that their
  sum \(M = \sum_{i=1}^{k}f_{i}(N_{i})\) satisfies the standard Hoeffding
  bounds\cite[Proposition~7]{Dubhashi1998}. For any \(s > 0\) denote \(b_{2} =
  \sqrt{s/m}\) so that we have
  \begin{equation*}
    \Prob(\mu' - b_{2} < M < \mu' + b_{2}) \geq 1 - 2e^{-2s}
  \end{equation*}
  where \(\mu' = \Expect(M)\).

  Let \(K \sim \PBin(f_{1}(N_{1}), \dotsc, f_{m}(N_{m}))\) conditional on
  \((N_{1}, \dotsc, N_{m})\). Note that \(\Expect(K \mid N) = M\), \(\Expect(K)
  = \Expect(M) = \mu'\) and that for any event \(A\) we have that
  \(\Expect(\Prob(\PBin(f_{1}(N_{1}), \dotsc, f_{m}(N_{m})) \in A)) = \Prob(K
  \in A)\). Letting \(b = b_{1} + b_{2}\) we have
  \begin{equation*}
    \begin{split}
      &\Prob(\mu' - b < K < \mu' + b) =\\
      &\quad= \Expect(\Prob(\mu' - b_{\phantom{1}} < K < \mu' + b_{\phantom{1}} \mid N))\\
      &\quad\geq \Expect(\Prob(\mu' - b_{\phantom{1}} < K < \mu' + b_{\phantom{1}} \mid N) \mid \mu' - b_{2} < M < \mu' + b_{2})) \\
      &\qquad\times (1 - 2e^{-2s}) \\
      &\quad\geq \Expect(\Prob(\makebox[\widthof{\(\mu' - b_{\phantom{1}} < K < \mu' + b_{\phantom{1}}\)}][c]{\(M - b_{1} < K < M + b_{1}\)} \mid N) \mid \mu' - b_{2} < M < \mu' + b_{2})) \\
      &\qquad\times (1 - 2e^{-2s}) \\
      &\quad\geq (1 - 2e^{-2t})(1 - 2e^{-2s})
    \end{split}
  \end{equation*}
  for any \(s, t > \log(\sqrt{2})\). Taking \(A = \N \cap (\mu' - b, \mu' + b)\)
  gives \(|A| \leq 2(\sqrt{\frac{t}{m}} + \sqrt{\frac{s}{m}})\). There must
  therefore exist some \(a \in A\) such that
  \begin{equation*}
  \Expect(\Prob(\PBin(f_{1}(N_{1}), \dotsc, f_{m}(N_{m})) = a)) \geq
    \frac{(1 - 2e^{-2t})(1 - 2e^{-2s})}{2(\sqrt{t} + \sqrt{s})}
    \frac{1}{\sqrt{m}}.
  \end{equation*}
  Taking \(s = t = 1.6\) yields
  \begin{equation*}
    \Expect(\Prob(\PBin(f_{1}(N_{1}), \dotsc, f_{m}(N_{m}) = a)) \geq \frac{1}{3}
    \frac{1}{\sqrt{m}}.\qedhere
\end{equation*}
\end{proof}

\section*{Acknowledgement}

I would like to thank Silvelyn Zwanzig for critical remarks on an earlier
version of this manuscript as well as Xing Shi Cai for discussions on how to
handle maxima of multinomial vectors.

\printbibliography{}

\appendix{}

\section{Basic facts and definitions}
\label{sec:basicfacts}

A statistical experiment \(\mathcal{E}\) is specified by a set \(\Theta\) called
the parameter space, a measurable space \(X\) that is the sample space of the
experiment, and a \(\Theta\)-indexed family of probability measures
\((P_{\theta})_{\theta \in \Theta}\) on \(X\).

For a measurable space \(X = (\mathcal{X}, \mathcal{A})\) the space
\(\mathcal{P}(X)\) of probability measures on \(X\) is a measurable space with
\(\sigma\)-algebras the coarsest such that for each \(A \in \mathcal{A}\) the
map \(P \mapsto P(A)\) is measurable.

For two measurable spaces \(X\) and \(Y\), any measurable \(f \colon X \to Y\)
lifts to the measurable push-forward \(f^{\ast} \colon \mathcal{P}(X) \to
\mathcal{P}(Y)\) by assigning to \(P \in \mathcal{P}(X)\) the composition \(A
\mapsto P(f^{-1}(A))\). The map has the property that if \(\eta\) is a
\(P\)-distributed \(X\)-valued random variable, then \(f(\eta)\) is a
\(f^{\ast}(P)\)-distributed \(Y\)-valued random variable.

The map \(f^{\ast}\) extends to a positive linear map from the space of signed
measures on \(X\) to the space of signed measures of \(Y\). The assignment is
functorial in the sense that for any pair of composable measurable functions
\(\beta_{1}\) and \(\beta_{2}\) one has \((\beta_{1} \circ \beta_{2})^{\ast} =
\beta_{1}^{\ast} \circ \beta_{2}^{\ast}\) and \(\id^{\ast} = \id\) for any
identity map. In particular, any inverse \(\beta^{-1}\) is sent to the
inverse of \(\beta^{\ast}\).

Given two measurable spaces \(X = (\mathcal{X}, \mathcal{A})\) and \(Y =
(\mathcal{Y}, \mathcal{B})\) a Markov-kernel \(K \colon X \to_{k} Y\) is a
measurable map from \(\mathcal{X}\) to the space of probability measures
\(\mathcal{P}(Y)\). This is equivalent to \(x \mapsto K(x)(B) = K_{x}(B)\) being
\(\mathcal{A}\)-measurable for every \(B \in \mathcal{B}\). For any measure
\(\mu\) on \(X\), define a@article{mccullagh2002statistical,
  title={What is a statistical model?},
  author={McCullagh, Peter},
  journal={The Annals of Statistics},
  volume={30},
  number={5},
  pages={1225--1310},
  year={2002},
  publisher={Institute of Mathematical Statistics}
}
 measure \(K(\mu) = K_{\mu}\) on \(Y\) by \(K_{\mu}(B)
= \int K_{x}(B) \, \mu(dx)\). If \(P \in \mathcal(P)(X)\) then \(K(P) \in
\mathcal{P}(Y)\). We will not distinguish in notation \(K\) from the map it
induces on the space of (probability) measures.

The direct sum (disjoint union) of two spaces \(X\) and \(Y\) is denoted by \(X
\oplus Y\) and has \(\sigma\)-algebra generated by \(\mathcal{A} \cup
\mathcal{B}\) after ensuring \(\mathcal{X}\) and \(\mathcal{Y}\) are disjoint.
For (signed) measures \(\mu\) and \(\nu\) on \(X\) and \(Y\), the (signed)
measure \(\mu \oplus \nu\) is defined by \((\mu \oplus \nu)(A) = \mu(A \cap
\mathcal{X}) + \nu(A \cap \mathcal{Y})\) for any measurable \(A\). Similarly, for
\(\mu\) on \(X\) and \(\nu\) on \(Y\) we denote by \(\mu \otimes \nu\) the
product measure specified by \((\mu \otimes \nu)(A \times B) = \mu(A)\nu(B)\).
The above generalise directly to a larger, possibly infinite, family of spaces
\((X_{i})_{i \in I}\) indexed by some set \(I\). The direct sum is then denoted
\(\bigoplus_{i \in I} X_{i}\) and the product \(\bigotimes_{i \in I} X_{i}\).
For each \(i \in I\) there is an associated injection \(\iota_{i} \colon X_{i}
\to \bigoplus_{i \in I} X_{i}\) such that \(\{ \iota_{i}(X_{i}) \mid i \in I\}\)
is a measurable partition of \(\bigoplus_{i \in I} X_{i}\) and \(\iota_{i}\)
restricts to a bimeasurable bijection onto \(\iota_{i}(X_{i})\).

For spaces \(X, X', Y, Y'\) and measurable maps \(f \colon X \to X'\) and \(g
\colon Y \to Y'\) we denote by \(f \otimes g \colon X \otimes Y \to X' \otimes
Y'\) and \(f \oplus g \colon X \oplus Y \to X' \oplus Y'\) the measurable maps
defined by \((f \otimes g)(x, y) = (f(x), g(y))\), \((f \otimes g)(\iota_{X}(x))
= \iota_{X'}(f(x))\), and \((f \otimes g)(\iota_{Y}(y)) = \iota_{Y'}(g(y))\). If
\(f\) and \(g\) are instead Markov kernels they are defined by \((f \otimes
g)_{(x, y)}(A \times B) = f_{x}(A)g_{y}(B)\) and \((f \oplus
g)_{\iota_{X}(x)}(\iota_{X'}(A)) = f_{x}(A)\) and \((f \oplus
g)_{\iota_{Y}(y)}(\iota_{Y'}(B)) = f_{y}(B)\).

Given any pair of experiments \(\mathcal{E}_{0} = (X, P_{\theta} ; \theta \in
\Theta)\) and \(\mathcal{E}_{1} = (Y, Q_{\theta} ; \theta \in \Theta)\) on the
same parameter space and any \(\alpha_{0} \in [0, 1]\) and denoting \(\alpha_{1}
= 1-\alpha_{0}\) this allows us to define the mixture experiment \(\sum_{i}
\alpha_{i} \mathcal{E}_{i} = (X \oplus Y, \alpha_{0} P_{\theta} \oplus
\alpha_{1} Q_{\theta} ; \theta \in \Theta)\). Similarly, the infinite convex
combination of some family of experiments \((\mathcal{E}_{i})_{i \in I}\) where
\(\mathcal{E}_{i} = (X_{i}, P_{i,\theta} ; \theta \in \Theta)\) with respect to
some coefficients \(p\) given by \(\sum_{i \in I}p_{i}\mathcal{E}_{i} =
(\bigoplus_{i \in I} X_{i}, \sum_{i \in I}\alpha_{i}P_{i,\theta}; \theta \in
\Theta)\). See Torgersen\cite[Chapter~1.3]{Torgersen1991} for more details.

For a parameter space \(\Theta\) a finite normalised decision problem on
\(\Theta\) is a tuple \((A, L)\) of a finite set \(A\) (the action space) and a
non-negative function \(L \colon \Theta \times A \to [0, 1]\). If \(L(\theta, a)
\in \{0, 1\}\) for each \(\theta \in \Theta\) and \(a \in A\) we say that \(L\)
is a 0-1-decision problem. Given a space \(X\) or experiment \(\mathcal{E} = (X,
P_{\theta} ; \theta \in \Theta)\) a decision procedure \(\rho\) for \((A, L)\)
on observing \(X\) (or \(\mathcal{E}\)) is a Markov kernel \(\rho \colon X
\to_{k} \mathcal{D}(A)\). Given a parameter space \(\Theta\) with a decision
problem \((A, L)\), an experiment \(\mathcal{E} = (X, P_{\theta} ; \theta \in
\Theta)\) and a decision procedure \(\rho\) on \(X\) the risk
\(R_{\mathcal{E}}(\rho, \theta)\) of \(\rho\) for \((L, A)\) at \(\theta \in
\Theta\) on observing \(\mathcal{E}\) is the expectation \(P_{\theta}(x \mapsto
L(\rho(x), \theta))\). Given a finitely supported probability measure \(\pi\) on
the parameter space \(\Theta\) the expected risk \(\sum_{\theta} \pi(\{\theta\})
R_{\mathcal{E}}(\rho, \theta)\) is called the Bayes risk of \(\rho\) for \((L,
A)\) with respect to the prior \(\pi\) on observing \(\mathcal{E}\). The infimum
over all \(\rho\) is the Bayes risk for \((L, A)\) with respect to the prior
\(\pi\) on observing \(\mathcal{E}\).

Given two experiments \(\mathcal{E} = (X, P_{\theta} ; \theta \in \Theta)\) and
\(\mathcal{F} = (Y, Q_{\theta} ; \theta \in \Theta)\) on the same parameter
space \(\Theta\) the deficiency \(\delta(\mathcal{E}, \mathcal{F})\) is the
smallest number \(\varepsilon \geq 0\) such that for any finite normalised
decision problem on \(\Theta\) and decision procedure \(\rho\) on \(Y\) there
exists a decision procedure \(\rho'\) on \(X\) such that for all \(\theta \in
\Theta\) one has \(R_{\mathcal{E}}(\rho', \theta) \leq R_{\mathcal{F}}(\rho,
\theta) + \varepsilon\). The deficiency can be characterised in various other
ways, most famously in terms of optimal transitions that map the family
\((P_{\theta})_{\theta \in \Theta}\) approximately onto the family
\((Q_{\theta})_{\theta \in \Theta}\) (see standard references
\cite{LeCam1986,Torgersen1991}). By construction, taking restrictions
\(\mathcal{E}'\) and \(\mathcal{F}'\) of \(\mathcal{E}\) and \(\mathcal{F}\) to
some subset \(\Theta' \subset \Theta\) reduces the deficiency:
\(\delta(\mathcal{E}', \mathcal{F}') \leq \delta(\mathcal{E}, \mathcal{F})\). If \(\delta(\mathcal{F}, \mathcal{E}) = 0\) we say \(\mathcal{E}\) is less
informative than \(\mathcal{F}\) or that \(\mathcal{F}\) is more informative
than \(\mathcal{E}\). If \(\mathcal{E}\) and \(\mathcal{F}\) are both less
informative than the other, they are equivalent and we write
\(\mathcal{E} \cong \mathcal{F}\). For any third experiment \(\mathcal{G}\) the
deficiency satisfies a triangle inequality \(\delta(\mathcal{E}, \mathcal{F})
\leq \delta(\mathcal{E}, \mathcal{G}) + \delta(\mathcal{G}, \mathcal{F})\). Being more/less informative is a transitive relation and
being equivalent is an equivalence relation and if \(\mathcal{E} \cong
\mathcal{F}\) then \(\delta(\mathcal{E}, \mathcal{G}) = \delta(\mathcal{F},
\mathcal{G})\) and \(\delta(\mathcal{G}, \mathcal{E}) = \delta(\mathcal{G},
\mathcal{F})\).

\end{document}